\documentclass[a4paper, reqno]{amsart}
\usepackage[utf8]{inputenc}

\usepackage[UKenglish]{babel}
\usepackage{amsmath, amssymb, amsthm}
\usepackage{scrextend}
\usepackage[hidelinks]{hyperref}
\usepackage{xpatch}
\usepackage{color}
\usepackage{tikz-cd}
\usepackage{enumitem}
\usepackage{theoremref}
\usepackage{ifthen}
\usepackage{xspace}

\usepackage[inline,final]{showlabels}
\showlabels{thlabel}

\setlist[description]{font=\normalfont\scshape}

\xpatchcmd{\proof}{\itshape}{\normalfont\bfseries}{}{}
\newtheoremstyle{repeat}{}{}{\itshape}{}{\bfseries}{.}{.5em}{#3, repeated}

\newtheorem{theorem}{Theorem}[section]
\newtheorem{proposition}[theorem]{Proposition}
\newtheorem{lemma}[theorem]{Lemma}
\newtheorem{corollary}[theorem]{Corollary}
\newtheorem{fact}[theorem]{Fact}

\theoremstyle{definition}
\newtheorem{definition}[theorem]{Definition}
\newtheorem{remark}[theorem]{Remark}

\newtheorem{example}[theorem]{Example}
\newtheorem{question}[theorem]{Question}

\theoremstyle{repeat}
\newtheorem*{repeated-theorem}{Repeat}

\renewcommand{\L}{\mathcal{L}}
\newcommand{\MM}{\mathfrak{M}}

\newcommand{\Q}{\mathbb{Q}}

\newcommand{\OP}{$\mathsf{OP}$\xspace}
\newcommand{\TP}[1][]{\ifthenelse{\equal{#1}{}}{$\mathsf{TP}$}{$\mathsf{TP_{#1}}$}\xspace}
\newcommand{\SOP}[1][]{\ifthenelse{\equal{#1}{}}{$\mathsf{SOP}$}{$\mathsf{SOP_{#1}}$}\xspace}
\newcommand{\IP}{$\mathsf{IP}$\xspace}
\newcommand{\NOP}{$\mathsf{NOP}$\xspace}
\newcommand{\NTP}[1][]{\ifthenelse{\equal{#1}{}}{$\mathsf{NTP}$}{$\mathsf{NTP_{#1}}$}\xspace}
\newcommand{\NSOP}[1][]{\ifthenelse{\equal{#1}{}}{$\mathsf{NSOP}$}{$\mathsf{NSOP_{#1}}$}\xspace}
\newcommand{\NIP}{$\mathsf{NIP}$\xspace}

\DeclareMathOperator{\tp}{tp}
\DeclareMathOperator{\Aut}{Aut}
\DeclareMathOperator{\cf}{cf}
\renewcommand{\S}{\operatorname{S}}

\renewcommand{\d}{\operatorname{d}}

\renewcommand{\phi}{\varphi}
\newcommand{\equivls}{\equiv^\textup{Ls}}
\newcommand{\op}{{\textup{op}}}

\def\Ind#1#2{#1\setbox0=\hbox{$#1x$}\kern\wd0\hbox to 0pt{\hss$#1\mid$\hss}
\lower.9\ht0\hbox to 0pt{\hss$#1\smile$\hss}\kern\wd0}
\def\ind{\mathop{\mathpalette\Ind{}}}
\def\Notind#1#2{#1\setbox0=\hbox{$#1x$}\kern\wd0\hbox to 0pt{\mathchardef
\nn="3236\hss$#1\nn$\kern1.4\wd0\hss}\hbox to 0pt{\hss$#1\mid$\hss}\lower.9\ht0
\hbox to 0pt{\hss$#1\smile$\hss}\kern\wd0}
\def\nind{\mathop{\mathpalette\Notind{}}}

\title{Dividing Lines between Positive Theories}
\author{Anna Dmitrieva, Francesco Gallinaro and Mark Kamsma}
\thanks{The first author was supported by Engineering and Physical Sciences Research Council Studentship. The second author was supported by EPSRC grant EP/S017313/1, by a London Mathematical Society Early Career Fellowship, and by the program GeoMod ANR-19-CE40-0022-01 (ANR-DFG). The third author was supported by EPSRC grant EP/W522314/1.}

\address[Anna Dmitrieva]{School of Mathematics, University of East Anglia, UK}
\email[]{a.dmitrieva@uea.ac.uk}
\urladdr[Anna Dmitrieva]{}

\address[Francesco Gallinaro]{Mathematisches Institut, Albert-Ludwigs-Universit\"at Freiburg, Germany}
\email[]{francesco.gallinaro@mathematik.uni-freiburg.de}
\urladdr[]{https://fgallinaro.github.io/}

\address[Mark Kamsma]{Department of Mathematics, Imperial College London, UK}
\email[]{mark@markkamsma.nl}
\urladdr[]{https://markkamsma.nl}

\date{\today}

\begin{document}

\maketitle

\begin{abstract}
We generalise the properties \OP, \IP, $k$-\TP, \TP[1], $k$-\TP[2], \SOP[1], \SOP[2] and \SOP[3] to positive logic, and prove various implications and equivalences between them. We also provide a characterisation of stability in positive logic in analogy with the one in full first-order logic, both on the level of formulas and on the level of theories. For simple theories there are the classically equivalent definitions of not having \TP and dividing having local character, which we prove to be equivalent in positive logic as well. Finally, we show that a thick theory $T$ has \OP iff it has \IP or \SOP[1] and that $T$ has \TP iff it has \SOP[1] or \TP[2], analogous to the well-known results in full first-order logic where \SOP[1] is replaced by \SOP in the former and by \TP[1] in the latter. Our proofs of these final two theorems are new and make use of Kim-independence.
\end{abstract}

\tableofcontents

\section{Introduction}
Model-theoretic dividing lines are used to measure how ``tame'' logical theories are. The most important such dividing lines can be formulated in terms of combinatorial properties. For example, a theory is stable if it does not have the order property. These various properties form an intricate diagram of implications and equivalences.

Positive logic is a generalisation of full first-order logic, and allows for the treatment of e.c.\ models of a non-companiable inductive theory \cite{haykazyan_existentially_2021}, hyperimaginaries (e.g.\ the $(-)^\text{heq}$ construction, see \cite[Subsection 10C]{dobrowolski_kim-independence_2022}), continuous logic \cite{ben-yaacov_model_2008} and more \cite{kamsma_bilinear_2023}. Some of these dividing lines have recently been studied in positive logic \cite{shelah_lazy_1975, pillay_forking_2000, ben-yaacov_simplicity_2003, haykazyan_existentially_2021, dobrowolski_kim-independence_2022, dobrowolski_amalgamation_2023}, and for some of them there is a positive version of the corresponding combinatorial property. However, these definitions and the implications between them that we know from full first-order logic are currently developed ad hoc, leaving gaps in the overall picture. For example, simplicity theory has been developed in positive logic \cite{pillay_forking_2000, ben-yaacov_simplicity_2003}, but simplicity in positive logic has so far only been defined in terms of local character for dividing and is nowhere equated to the usual definition of not having the tree property. The main goal of this paper is to provide the definitions of and implications between the most important dividing lines in terms of combinatorial properties, while also proving equivalences with other characterisations of these dividing lines.

\textbf{Main results.}
In full first-order logic stable formulas are characterised in various ways, for example as those that do not have have \OP (the \emph{order property}) or by counting types. We recover this characterisation in positive logic in \thref{thm:stable-formula}, tying together previous work on stability in positive logic from \cite{shelah_lazy_1975, ben-yaacov_simplicity_2003}. Subsequently, we obtain the usual equivalence of definitions for a stable theory in \thref{stable-theory}: either through type counting or by the lack of \OP.

Our first main result in the unstable setting is the implication diagram between the various combinatorial properties we consider. See the start of Section \ref{sec:implication-between-the-properties} for a discussion about the strictness of implications, and implications that are missing compared to full first-order logic.
\begin{theorem}
\thlabel{implication-diagram}
The following implications between properties hold for a positive theory $T$.
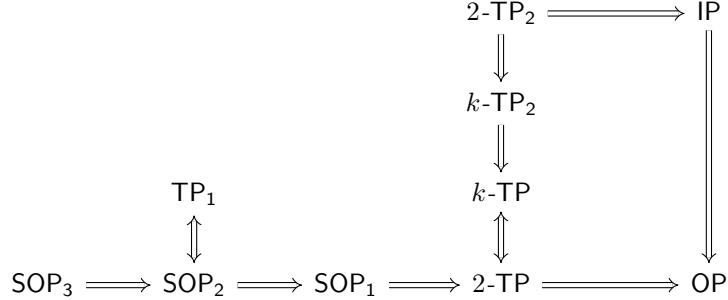
\[
% https://tikzcd.yichuanshen.de/#N4Igdg9gJgpgziAXAbVABwnAlgFyxMJZAZgBoBGAXVJADcBDAGwFcYkQBrAHS5xgA8cwALQ8AKgAVkAJkoBfEHNLpMufIRRkADNTpNW7Hn0GQATgFsmwaaK6SZ8xcpAZseAkQCspHTQYs2RBAjASEeAEkJBSUVN3UiclJZP31A4N5Q4HEpKmjnVzUPTSTdfwMg7gzBEWy82MKNZC1SYlLUwyqwrgBlAHkpVrqXVXdGxNaUgI7jLr6pWSGC0aJpFrapoJDqnjnkXKd65eKJvQ30mbNLRmtbSUWR+JRvE7K0ra7+6N0YKABzeCIoAAZqYIOYkIkQDgIEhmqdymBmIxGDRGPQAEYwRgSB5FECmLC-AAWOBAqKwYDSUAgzHRjDYNCJMHoUCQiORNBw9CwjHYkEpBxAILBsM5MMQABZJgikSiQGjMdjcRp5TAgaTyQKgtTafTBcLwYhIdCkKt4YF2XKFVicXE8fT1WT5RSqTS6WwYkLQYazSbEABOaUW2WojE25XsB0a51akA693671IMhQ8UANiDbJD8rDSrtKqjiZFkrFSAAHJnEJbQ4rbQ1I2rSZ6DUhvKmkBnzVmOTnaxGgoXNa7dR7nC3EJ2-QB2SvV3vh-MNx1D9jxvXNpOIGftxAVrtV7PWvP1geNp2MF2rt3rsebvd+wP7udHutHVXLmPDhNyeRAA
\begin{tikzcd}
                                       &                                        &                                        & {2\text{-\TP[2]}} \arrow[d, Rightarrow] \arrow[rr, Rightarrow] &  & \text{\IP} \arrow[ddd, Rightarrow] \\
                                       &                                        &                                        & {k\text{-\TP[2]}} \arrow[d, Rightarrow]                              &  &                                    \\
                                       & {\text{\TP[1]}} \arrow[d, Leftrightarrow]              &                                        & k\text{-\TP} \arrow[d, Leftrightarrow]                                               &  &                                    \\
{\text{\SOP[3]}} \arrow[r, Rightarrow] & {\text{\SOP[2]}} \arrow[r, Rightarrow] & {\text{\SOP[1]}} \arrow[r, Rightarrow] & 2\text{-\TP} \arrow[rr, Rightarrow]                            &  & \text{\OP}                        
\end{tikzcd}
\]
\end{theorem}

Like stable theories, simple theories can be defined in different ways, which are equivalent in full first-order logic. This includes defining  simplicity in terms of local character for dividing, as is done in previous studies of simplicity in positive logic \cite{ben-yaacov_simplicity_2003}, or as those theories not having \TP. We prove that these are equivalent in positive logic as well in \thref{dividing-local-character-iff-ntp}.

Finally, we recall the following two famous theorems from full first-order logic. Here \SOP stands for the \emph{strict order property}, a property that we do not consider in this paper but implies \SOP[3] (see also \thref{rem:sop}).
\begin{theorem}[{\cite[Theorem II.4.7]{shelah_classification_1990}}]
\thlabel{thm:first-order-op-iff-ip-or-sop}
A full first-order theory $T$ has \OP iff it has \IP or \SOP.
\end{theorem}
\begin{theorem}[{\cite[Theorem III.7.11]{shelah_classification_1990}\footnote{Gaps in this proof have been filled in in \cite[Theorem 5.9]{kim_tree_2014}.}}]
\thlabel{thm:first-order-tp-iff-tp1-or-tp2}
A full first-order theory $T$ has \TP iff it has \TP[1] or \TP[2].
\end{theorem}
We will prove the following versions of these theorems for positive logic.
\begin{theorem}
\thlabel{thm:op-iff-ip-or-sop1}
A thick theory $T$ has \OP iff it has \IP or \SOP[1]. Equivalently: $T$ is stable iff it is \NIP and \NSOP[1].
\end{theorem}
\begin{theorem}
\thlabel{thm:tp-iff-sop1-or-tp2}
A thick theory $T$ has \TP iff it has \SOP[1] or \TP[2]. Equivalently: $T$ is simple iff it is \NSOP[1] and \NTP[2].
\end{theorem}
For an in-depth discussion about why we use \NSOP[1] we refer to \thref{rem:two-famous-theorems-comparison}. It is worth mentioning however that our proofs are completely different from the proofs of the original two theorems. Using recent developments on Kim-independence in \NSOP[1] theories we give proofs based on independence relations. Thickness is a mild assumption that is automatically satisfied in full first-order logic, see also \thref{def:thickness} and the discussion before it.

\textbf{Overview.} We start with the basics for positive logic in Section \ref{sec:preliminaries}. We deal with the different characterisations of stable formulas and stable theories in Section \ref{sec:positive-stability}. We then collect all the definitions of the various combinatorial properties we consider in Section \ref{sec:definitions}. In Section \ref{sec:implication-between-the-properties} we prove the implications between the various properties, i.e.\ we prove \thref{implication-diagram}. In Section \ref{sec:interactions-with-independence-relations} we consider interactions between independence relations and some of the combinatorial properties, obtaining the equivalence of definitions for a simple theory and proving \thref{thm:op-iff-ip-or-sop1} and \thref{thm:tp-iff-sop1-or-tp2}. Finally, Section \ref{sec:discussion-and-open-questions} discusses and asks some natural questions.

\textbf{Acknowledgements.} We would like to thank Jonathan Kirby for many useful discussions. We would also like to thank the anonymous referee for their comments that helped improve this paper.

\section{Preliminaries of positive logic}
\label{sec:preliminaries}
We only recall the definitions and facts about positive logic that we need, for a more extensive treatment and discussion see \cite{ben-yaacov_positive_2003, poizat_positive_2018} and for a more survey-like overview see \cite[Section 2]{dobrowolski_kim-independence_2022}.
\begin{definition}
\thlabel{def:positive-syntax}
Fix a signature $\L$. A \emph{positive formula} in $\L$ is one that is obtained from combining atomic formulas using $\wedge$, $\vee$, $\top$, $\bot$ and $\exists$. An \emph{h-inductive sentence} is a sentence of the form $\forall x(\phi(x) \to \psi(x))$, where $\phi(x)$ and $\psi(x)$ are positive existential formulas. A \emph{positive theory} is a set of h-inductive sentences.
\end{definition}
Whenever we say ``formula'' or ``theory'' we will mean ``positive formula'' and ``positive theory'' respectively, unless explicitly stated otherwise. This also means that every formula and theory we consider will be implicitly assumed to be positive.
\begin{remark}
\thlabel{rem:morleyisation}
We can study full first-order logic as a special case of positive logic. This is done through a process called \emph{Morleyisation}. For this we add a relation symbol $R_\phi(x)$ to our language for every full first-order formula $\phi(x)$. Then we have our theory (inductively) express that $R_\phi(x)$ and $\phi(x)$ are equivalent. This way every first-order formula is (equivalent to) a relation symbol, and thus in particular to a positive existential formula.
\end{remark}
We are generally only interested in \emph{existentially closed} models. These can be characterised in various ways, but the one that matters for us is the following.
\begin{definition}
\thlabel{def:negation}
A \emph{negation} of a formula $\phi(x)$ is a formula $\psi(x)$ such that $T \models \neg \exists x(\phi(x) \wedge \psi(x))$. Equivalently, $\psi(x)$ implies $\neg \phi(x)$ modulo $T$.
\end{definition}
\begin{definition}
\thlabel{def:existentially-closed}
We call a model $M$ of a theory $T$ \emph{existentially closed} or \emph{e.c.}\ if whenever $M \not \models \phi(a)$ then there is a negation $\psi(x)$ of $\phi(x)$ with $M \models \psi(a)$
\end{definition}
Following our earlier convention about dropping the ``positive'' everywhere, a (positive) \emph{type} will be a set of (positive) formulas, over some parameter set $B$, satisfied by some tuple $a$ in some e.c.\ model $M$:
\[
\tp(a/B) = \{ \phi(x, b) : M \models \phi(a, b) \text{ and } b \in B \}.
\]
Throughout we will assume that our theories have the \textit{joint continuation property} or JCP (that is, for any two models $M_1$ and $M_2$ there is a model $N$ with homomorphisms $M_1 \rightarrow N \leftarrow M_2$). This is the positive version of working in a complete theory, and we can always extend a theory $T$ to a theory with JCP by taking the set of all h-inductive sentences that are true in some e.c.\ model of $T$. Under the JCP assumption we can work in a monster model, and these can be constructed for positive theories using the usual techniques. We let the reader fix their favourite notion of smallness (e.g., fix a big enough cardinal $\kappa$, and let ``small'' mean $< \kappa$). We recall the properties of a monster model $\MM$:
\begin{itemize}
\item \emph{existentially closed}, $\MM$ is an e.c.\ model;
\item \emph{very homogeneous}, for any small $a, b, C$ we have $\tp(a/C) = \tp(b/C)$ iff there is $f \in \Aut(\MM/C)$ with $f(a) = b$ (we will also write $a \equiv_C b$);
\item \emph{very saturated}, any finitely satisfiable small set of formulas $\Sigma$ over $\MM$ is satisfiable in $\MM$.
\end{itemize}
As usual, we will omit the monster model from notation. For example, we write $\models \phi(a)$ instead of $\MM \models \phi(a)$.

We finish this section with the definition of indiscernible sequences and a lemma to find such sequences. The construction of indiscernible sequences using Ramsey's theorem fails in positive logic, but the construction using the Erd\H{o}s-Rado theorem goes through and gives in fact a stronger result.
\begin{definition}
\thlabel{def:indiscernible-sequences}
A sequence $(a_i)_{i \in I}$ (for some linear order $I$) is $C$-indiscernible if for any $i_1 < \ldots < i_n$ and $j_1 < \ldots < j_n$ in $I$ we have $a_{i_1} \ldots a_{i_n} \equiv_C a_{j_1} \ldots a_{j_n}$.
\end{definition}
\begin{lemma}[{\cite[Lemma 1.2]{ben-yaacov_simplicity_2003}}]
\thlabel{erdos-rado-indiscernible-sequences}
Let $C$ be any parameter set, $\kappa$ any cardinal, and let $\lambda = \beth_{(2^{|T| + |C| + \kappa})^+}$. Then for any sequence $(a_i)_{i < \lambda}$ of $\kappa$-tuples there is a $C$-indiscernible sequence $(b_i)_{i < \omega}$ such that for all $n < \omega$ there are $i_1 < \ldots < i_n < \lambda$ with $b_1 \ldots b_n \equiv_C a_{i_1} \ldots a_{i_n}$.
\end{lemma}

\begin{definition}
For a theory $T$ we write $\lambda_T = \beth_{(2^{|T|})^+}$.
\end{definition}

\begin{remark}
\thlabel{rem:boundedness}
Since inequality may not be positively definable, there may be infinite bounded positively definable sets in our monster. In fact, the cardinality of the e.c.\ models of a positive theory might be bounded (such a theory is called \emph{bounded}), which results in a monster model that is itself ``small''. An extreme example is the empty theory in the empty language, whose e.c.\ models are singletons, and so the monster is a singleton. However, there is no need for special treatment for these cases. It just means that if we speak about a sequence (or otherwise indexed set) of parameters $(a_i)_{i < \lambda}$ where $\lambda$ is larger than the cardinality of the monster, we will have duplicates in this sequence. Particularly, the only indiscernible sequences in bounded theories (or, more generally, in bounded positively definable sets) are the constant ones.
\end{remark}

\section{Positive stability}
\label{sec:positive-stability}
In this section we begin our treatment of dividing lines in positive theories from \emph{stability}. We introduce the \emph{order property} (\thref{def:order-property}), the first example of the combinatorial properties which will be discussed in the next sections. \thref{thm:stable-formula} provides a characterisation of stable formulas in the positive context, analogous to the various characterising properties that are well known from full first-order logic. The techniques used in this section are adapted from \cite{shelah_finite_1970, shelah_lazy_1975, grossberg_shelahs_2002, ben-yaacov_simplicity_2003}, as well as from the standard techniques used for full first-order theories. There is also work on stability in the positive setting in \cite[Chapter 4]{belkasmi_contributions_2012}, see \thref{rem:belkasmi-thesis} for more details.
\begin{definition}
\thlabel{def:phi-type}
Let $\phi(x, y)$ be a formula. For $a$ and a parameter set $B$, we write
\[
\tp_\phi(a/B) = \{ \phi(x, b) : \, \models \phi(a, b) \text{ where } b \in B \}.
\]
A \emph{$\phi$-type} over $B$ is a set of formulas of the form $\tp_\phi(a/B)$ for some $a$. So it is the restriction of a maximal type over $B$ to just the $\phi$-formulas. We write $\S_\phi(B)$ for the set of $\phi$-types over $B$.
\end{definition}
\begin{example}
\thlabel{ex:phi-type-not-maximal}
A $\phi$-type is not necessarily maximal. For example, consider the theory $T$ with inequality and two disjoint unary predicates $P$ and $Q$. The e.c.\ models of $T$ are then simply two disjoint infinite sets. Let $M$ be such an e.c.\ model and let $a \in P(M)$ and $b \in Q(M)$. Let $\phi(x)$ be the formula $P(x)$: then $\tp_\phi(a) = \{ \phi(x) \}$, while $\tp_\phi(b) = \emptyset$.
\end{example}
\begin{definition}
\thlabel{def:stable-formula}
Let $\lambda$ be an infinite cardinal. A formula $\phi(x, y)$ is \emph{$\lambda$-stable} if $|B| \leq \lambda$ implies $|\S_\phi(B)| \leq \lambda$. We call $\phi(x, y)$ \emph{stable} if it is $\lambda$-stable for some $\lambda$.
\end{definition}
The following is taken from \cite[Definition 2.1]{ben-yaacov_simplicity_2003}.
\begin{definition}
\thlabel{def:definable-phi-type}
Let $p(x)$ be a type over $B$ and let $\phi(x, y)$ be a formula. A \emph{$\phi$-definition} of $p(x)$ over $C$ is a partial type $\d_p \phi(y)$ over $C$ with $|\d_p \phi(y)| \leq |T|$ such that
\[
\phi(x, b) \in p(x) \quad \Longleftrightarrow \quad \models \d_p \phi(b).
\]
We say that $p(x)$ is \emph{$\phi$-definable (over $C$)} if it has a $\phi$-definition over $C$. If $p(x)$ is $\phi$-definable over $B$ we just say it is $\phi$-definable.
\end{definition}
\begin{definition}
\thlabel{def:order-property}
A formula $\phi(x,y)$ has the \emph{order property} (\OP) if there are sequences $(a_i)_{i < \omega}$ and $(b_i)_{i < \omega}$ and a negation $\psi(x,y)$ of $\phi(x,y)$ such that for all $i, j < \omega$ we have: 
\begin{align*}
&\models \phi(a_i, b_j) \quad\text{if} \ i < j, \\
&\models \psi(a_i, b_j) \quad\text{if} \ i \geq j.
\end{align*}
\end{definition}
Note that by compactness the exact shape of the linear order in the order property (\thref{def:order-property}) does not matter. That is, we can replace $\omega$ with any infinite linear order. In fact, we can use this trick to state the order property in terms of indiscernible sequences, getting rid of the negation $\psi(x, y)$.
\begin{proposition}
\thlabel{prop:order-property-indiscernible-sequence}
A formula $\phi(x, y)$ has the order property iff there is an indiscernible sequence $(a_i b_i)_{i < \omega}$ such that
\[
\models \phi(a_i, b_j) \quad \Longleftrightarrow \quad i < j.
\]
\end{proposition}
\begin{proof}
For the left to right direction let $(a'_i)_{i < \omega}$, $(b'_i)_{i < \omega}$ and $\psi(x, y)$ witness the order property. By compactness we may elongate the sequences to $(a'_i)_{i < \lambda}$ and $(b'_i)_{i < \lambda}$. Making sure that $\lambda$ is big enough, we can then by \thref{erdos-rado-indiscernible-sequences} base an indiscernible sequence $(a_i b_i)_{i < \omega}$ on $(a'_i b'_i)_{i < \lambda}$. Now if $i < j < \omega$ then there are $i_0 < j_0 < \lambda$ such that $a_i b_j \equiv a'_{i_0} b'_{j_0}$, and so $\models \phi(a_i, b_j)$ follows from $\models \phi(a'_{i_0}, b'_{j_0})$. For the converse we prove the contrapositive, so let $j \leq i$. Then there are $j_0 \leq i_0 < \lambda$ (with $j_0 = i_0$ iff $j = i$) such that $a_i b_j \equiv a'_{i_0} b'_{j_0}$. Hence $\models \psi(a'_{i_0}, b'_{j_0})$ and so $\not \models \phi(a_i, b_j)$.

For the right to left direction we only need to find the negation $\psi(x, y)$. As we have $\not \models \phi(a_0, b_0)$ there must be some negation $\psi_1(x, y)$ of $\phi(x, y)$ with $\models \psi_1(a_0, b_0)$. By indiscernibility we have $\models \psi_1(a_i, b_i)$ for all $i < \omega$. Similarly, using $\not \models \phi(a_1, b_0)$ we find a negation $\psi_2(x, y)$ with $\models \psi_2(a_i, b_j)$ for all $j < i$. Take $\psi(x, y)$ to be $\psi_1(x, y) \vee \psi_2(x, y)$. As both of $\psi_1(x, y)$ and $\psi_2(x, y)$ are negations of $\phi(x, y)$ we have that $\psi(x, y)$ is also a negation of $\phi(x, y)$. Furthermore, by construction $j \leq i$ implies $\models \psi(a_i, b_j)$.
\end{proof}
\begin{definition}
\thlabel{def:binary-tree-property}
A formula $\phi(x, y)$ is said to have the \emph{binary tree property} if there is a negation $\psi(x, y)$ of $\phi(x, y)$ together with $(b_\eta)_{2^{< \omega}}$ such that for every $\sigma \in 2^\omega$ the set
\[
\{ \chi_{\sigma(n)}(x, b_{\sigma|_n}) : n < \omega \}
\]
is consistent, where $\chi_0 := \phi$ and $\chi_1 := \psi$.
\end{definition}
\begin{definition}[{\cite[Definition 2.1]{ben-yaacov_simplicity_2003}, simplified}]
\thlabel{def:phi-psi-rank}
For contradictory formulas $\phi(x, y)$ and $\psi(x, y)$ we define the \emph{$(\phi, \psi)$-rank} $R_{\phi, \psi}(-)$ as follows. The input is a set of formulas (possibly with parameters) in free variables $x$. Then $R_{\phi, \psi}(-)$ is the least function into the ordinals (together with $-1$ and $\infty$) such that:
\begin{itemize}
\item $R_{\phi, \psi}(\Sigma) \geq 0$ if $\Sigma(x)$ is consistent;
\item $R_{\phi, \psi}(\Sigma) \geq \alpha + 1$ if there is some $b$ such that $R_{\phi, \psi}(\Sigma \cup \{\phi(x, b)\}) \geq \alpha$ and $R_{\phi, \psi}(\Sigma \cup \{\psi(x, b)\}) \geq \alpha$;
\item $R_{\phi, \psi}(\Sigma) \geq \ell$ if $R_{\phi, \psi}(\Sigma) \geq \alpha$ for all $\alpha < \ell$, where $\ell$ is a limit ordinal.
\end{itemize}
\end{definition}
\begin{lemma}
\thlabel{lem:phi-psi-rank-basics}
Let $\phi(x, y)$ and $\psi(x, y)$ be contradictory formulas.
\begin{enumerate}[label=(\roman*)]
\item If $\Sigma(x)$ implies $\Sigma'(x)$ then $R_{\phi, \psi}(\Sigma) \leq R_{\phi, \psi}(\Sigma')$.
\item The property $R_{\phi, \psi}(\Sigma) \geq n$ is type-definable by
\[
\exists (y_\eta)_{\eta \in 2^{< n}} \left( \bigwedge_{\sigma \in 2^n} \exists x \left( \Sigma(x) \wedge \bigwedge_{k < n} \chi_{\sigma(k)}(x, y_{\sigma|_k}) \right) \right),
\]
where $\chi_0$ and $\chi_1$ are $\phi$ and $\psi$ respectively. In particular, if $\Sigma$ is finite (i.e.\ a formula), then this is just a formula.
\end{enumerate}
\end{lemma}
\begin{proof}
Both are straightforward induction arguments. The key intuition being that $R_{\phi, \psi}(\Sigma) \geq n$ expresses that we can build a binary tree like \thref{def:binary-tree-property} of height $n$ and where every path is also consistent with $\Sigma$.
\end{proof}
\begin{lemma}
\thlabel{lem:binary-tree-iff-infinite-phi-psi-rank}
A formula $\phi(x, y)$ has the binary tree property iff there is a negation $\psi(x, y)$ of $\phi(x, y)$ such that $R_{\phi, \psi}(x=x) \geq \omega$.
\end{lemma}
\begin{proof}
By \thref{lem:phi-psi-rank-basics} and compactness.
\end{proof}
\begin{theorem}
\thlabel{thm:stable-formula}
The following are equivalent for a formula $\phi(x, y)$:
\begin{enumerate}[label=(\roman*)]
\item $\phi$ is stable,
\item $|\S_\phi(B)| \leq (|B| + |T|)^{|T|}$ for every $B$,
\item $\phi$ does not have the order property,
\item $\phi$ does not have the binary tree property,
\item $R_{\phi,\psi}(x=x) < \omega$ for every negation $\psi(x, y)$ of $\phi(x, y)$,
\item for any $B$ every type over $B$ is $\phi$-definable.
\end{enumerate}
\end{theorem}
\begin{proof}
The equivalence (iv) $\Leftrightarrow$ (v) is \thref{lem:binary-tree-iff-infinite-phi-psi-rank}. The equivalence between (i), (ii), (v) and (vi) is exactly \cite[Proposition 2.2]{ben-yaacov_simplicity_2003}.

\underline{(i) $\Rightarrow$ (iii)} We prove the contraposition. So let $\lambda$ be an arbitrary infinite cardinal. By a standard result there is a linear order $I$ with a dense subset $I_0 \subseteq I$ such that $|I_0| = \lambda$ and $|I| > \lambda$ (see e.g.\ \cite[Exercise 8.2.8]{tent_course_2012}). Let $\psi(x, y)$ be the negation of $\phi(x, y)$ witnessing the order property. So by compactness there are $(a_i)_{i \in I}$ and $(b_i)_{i \in I}$ such that for all $i,j \in I$:
\begin{align*}
&\models \phi(a_i, b_j) & \text{if } i < j,\\
&\models \psi(a_i, b_j) & \text{if } i \geq j.
\end{align*}
Set $B = (b_i)_{i \in I_0}$, then as $I_0$ is dense in $I$ we have that $\tp_\phi(a_i/B) \neq \tp_\phi(a_j/B)$ for any $i \neq j$. So we find $|\S_\phi(B)| \geq |I| > \lambda$ while $|B| \leq \lambda$ and we conclude that $\phi$ is not $\lambda$-stable.

\underline{(iii) $\Rightarrow$ (i)} This implication requires some more preparation, so we postpone it to \thref{lem:nop-implies-stable}.
\end{proof}
\begin{example}
\thlabel{ex:counterexample-to-classical-stable-formula}
In full first-order logic, for a formula $\phi(x, y)$ the following are equivalent (see e.g.\ \cite[Theorem 8.2.3]{tent_course_2012}):
\begin{enumerate}[label=(\roman*)]
\item $\phi$ is stable,
\item there is no sequence $(a_i b_i)_{i < \omega}$ such that $\models \phi(a_i, b_j)$ iff $i < j$,
\item $|\S_\phi(B)| \leq |B|$ for any infinite $B$.
\end{enumerate}
Of course, (ii) is the classical formulation of the order property. In a full first-order theory this is easily seen to be equivalent to \thref{def:order-property}: just take $\psi(x, y)$ to be $\neg \phi(x, y)$. Point (iii) is a stronger version of \thref{thm:stable-formula}(ii).

We will show that this equivalence generally fails in positive logic. That is, we will construct a theory together with a stable formula $\phi(x, y)$ (in fact, the entire theory will be stable) such that (ii) and (iii) fail for $\phi$.

Write $\Q_{(0,1)} = \{q \in \Q : 0 < q < 1\}$. Consider the language $\L$ with a constant for each element of $\Q_{(0,1)}$, and an order symbol $\leq$. Considering the obvious $\L$-structure on $\Q_{(0,1)}$, we let $T$ be the set of all h-inductive sentences true in $\Q_{(0,1)}$. One quickly verifies that the real unit interval $[0, 1]$ is a maximal e.c.\ model for this theory. So the number of $\phi$-types is bounded by $2^{\aleph_0}$, for any $\phi$. Hence every formula is stable.

Consider the formula $\phi(x, y)$ given by $x \leq y$. For $n < \omega$ set $a_n = 1-\frac{1}{n+2}$ and $b_n = 1-\frac{1}{n+1}$. Then clearly $\models \phi(a_i, b_j)$ iff $i < j$, so (ii) fails for $\phi(x, y)$. The important difference with \thref{def:order-property} is of course that for $i \geq j$ there is not just one uniform reason (in the form of a negation of $\phi$) for $\not \models \phi(a_i, b_j)$.

Using the same formula $\phi(x, y)$, we let $B = \Q_{(0,1)}$. The $\phi$-types over $B$ then correspond exactly to real numbers in $[0,1]$, via Dedekind cuts. So we have $|\S_\phi(B)| = 2^{\aleph_0} > \aleph_0 = |B|$, and hence (iii) fails As $B$ only contains constants we may even take $B = \emptyset$, but then $B$ is no longer infinite, which is technically required for (iii).

Generally, this example shows that in positive logic we may find some infinite linear order in a stable theory, but as long as they are bounded this should not cause unstability. Intuitively this is because growth (e.g.\ of the type spaces) beyond that bound is then again well-behaved.

Note that in particular this sort of behaviour can also appear in unbounded theories, if they have bounded sorts or bounded positively definable sets. For example, we could add a separate sort with a symbol for inequality to the theory in this example, and have our theory state that the additional sort is an infinite set. The theory is now unbounded, but the example still goes through.
\end{example}
With the adjusted definitions for stability of a formula, we get the usual equivalent definitions of a stable theory. The arguments are standard, but we include them for completeness' sake.
\begin{definition}
\thlabel{def:stable-theory}
Let $\lambda$ be an infinite cardinal. A theory $T$ is \emph{$\lambda$-stable} if $|B| \leq \lambda$ implies $|\S_n(B)| \leq \lambda$ for all $n < \omega$, where $\S_n(B)$ is the set of $n$-types with parameters in $B$. We call $T$ \emph{stable} if it is $\lambda$-stable for some $\lambda$. 
\end{definition}
\begin{example}
\thlabel{ex:bounded-stable}
Any bounded theory is stable: since every type must be realised in the monster, we have for all $n < \omega$ that $|S_n(\MM)|=|\MM|$. Considering $|\MM|$ is fine here, because in bounded theories the monster is small, see \thref{rem:boundedness}.
\end{example}
\begin{theorem}
\thlabel{stable-theory}
The following are equivalent for a theory $T$:
\begin{enumerate}[label=(\roman*)]
\item $T$ is stable,
\item all formulas in $T$ are stable,
\item $T$ is $\lambda$-stable for all $\lambda$ such that $\lambda^{|T|} = \lambda$.
\end{enumerate}
\end{theorem}
\begin{proof}
\underline{(i) $\Rightarrow$ (ii)} Let $\lambda$ be such that $T$ is $\lambda$-stable. Then whenever $|B| \leq \lambda$ we have for any $\phi(x, y)$ that $|\S_\phi(B)| \leq |\S_n(B)| \leq \lambda$, where $n = |x|$. So every formula is $\lambda$-stable.

\underline{(ii) $\Rightarrow$ (iii)} Let $\lambda$ be such that $\lambda^{|T|} = \lambda$, and let $|B| \leq \lambda$. As $\lambda^{|T|} = \lambda$ we have that $\lambda > |T|$. So for any $\phi$ we have by \thref{thm:stable-formula} that $|\S_\phi(B)| \leq (|B| + |T|)^{|T|} \leq \lambda^{|T|} = \lambda$. Every type is fully determined by its restrictions to $\phi$-types, as $\phi$ ranges over all formulas in the theory. So there are at most $|T| \times \lambda = \lambda$ many types over $B$, as required.

\underline{(iii) $\Rightarrow$ (i)} Note that $(2^{|T|})^{|T|} = 2^{|T|}$, so $T$ is $2^{|T|}$-stable and hence stable.
\end{proof}
In the remainder of this section we finish the proof of \thref{thm:stable-formula}.
\begin{definition}
\thlabel{def:splitting-type}
Let $\phi(x, y)$ and $\psi(y, z)$ be formulas without parameters and let $A \subseteq B$ be sets of parameters. We say that a type $p(x) \in \S_\phi(B)$ \emph{$(\psi, \phi)$-splits over $A$} if there are $b, b' \in B$ such that $\tp_\psi(b/A) = \tp_\psi(b'/A)$ while $\phi(x, b) \in p(x)$ and $\phi(x, b') \not \in p(x)$.
\end{definition}
\begin{lemma}
\thlabel{lem:non-splitting-facts}
Let $\phi(x, y)$ and $\psi(y, z)$ be formulas without parameters, and let $A \subseteq C$ be parameter sets.
\begin{enumerate}[label=(\roman*)]
\item Suppose $B$ is such that $A \subseteq B \subseteq C$ and it realises every $\psi$-type over $A$ that is realised in $C$. Then if $p_1, p_2 \in \S_\phi(C)$ do not $(\psi, \phi)$-split over $A$ we have that $p_1|_B = p_2|_B$ implies $p_1 = p_2$.
\item There are at most $2^{|\S_\psi(A)| + |A| + |T|}$ many types in $\S_\phi(C)$ that do not $(\psi, \phi)$-split over $A$.
\item If $\lambda \geq |A| + |T|$ then there are at most $2^{2^\lambda}$ many types in $\S_\phi(C)$ that do not $(\psi, \phi)$-split over $A$.
\end{enumerate}
\end{lemma}
\begin{proof}
To prove (i) we show that $p_1 \subseteq p_2$, from which the result follows by symmetry. Let $\phi(x, c) \in p_1$. By the assumption on $B$, $\tp_\psi(c/A)$ is realised by some $b \in B$. As $p_1$ does not $(\psi, \phi)$-split over $A$ we must then have $\phi(x, b) \in p_1$. We thus have $\phi(x, b) \in p_2$, because $p_1|_B = p_2|_B$, and $\phi(x, c) \in p_2$ follows from the fact that $p_2$ does not $(\psi, \phi)$-split over $A$.

For (ii) we can let $B$ be such that $A \subseteq B \subseteq C$ and realising every $\psi$-type over $A$ that is realised in $C$, while also $|B| \leq |\S_\psi(A)| + |A|$. By (i) then the number of types in $\S_\phi(C)$ that do not $(\psi, \phi)$-split over $A$ is bounded by $|\S_\phi(B)| \leq 2^{|B| + |T|} \leq 2^{|\S_\psi(A)| + |A| + |T|}$.

Finally, for (iii) we apply (ii) using that $|\S_\psi(A)| + |A| + |T| \leq 2^\lambda$.
\end{proof}
We can now fill in the final missing piece of \thref{thm:stable-formula}. The proof strategy used here is based on \cite{grossberg_shelahs_2002}.
\begin{lemma}
\thlabel{lem:nop-implies-stable}
If a formula $\phi(x, y)$ does not have the order property then it is stable.
\end{lemma}
\begin{proof}
We prove the contrapositive, so we assume that $\phi(x, y)$ is not stable. For convenience, set $\mu = 2^{2^{\lambda_T}}$. As $\phi$ is not $\mu$-stable, we find some set $A$ such that $|A| \leq \mu$ and $\S_\phi(A) > \mu$. We can thus find $(a_i)_{i < \mu^+}$ such that $\tp_\phi(a_i / A) \neq \tp_\phi(a_j / A)$ for all $i \neq j < \mu^+$. We inductively build a continuous chain of sets $(A_i)_{i < \mu}$ with $A_0 = A$ such that for all $i < \mu$:
\begin{enumerate}[label=(A\arabic*)]
\item $|A_i| \leq \mu$,
\item for every $B \subseteq A_i$ with $|B| \leq \lambda_T$ every type in $\S(B)$ (in finitely many variables) is realised in $A_{i+1}$.
\end{enumerate}
We can indeed do this because there are at most $\mu^{\lambda_T} = \mu$ many subsets of $A_i$ that have cardinality at most $\lambda_T$, and there are at most $2^{|B| + |T|} \leq 2^{\lambda_T} < \mu$ many types over such a parameter set $B$.

Set $\chi(y, x) := \phi(x, y)$. We now claim that there are cofinally many $i < \mu^+$ such that for all $j < \lambda_T$ the type $\tp_\phi(a_i / A_j)$ $(\chi, \phi)$-splits over each $B \subseteq A_j$ of cardinality at most $\lambda_T$.

\emph{Proof of claim.} Suppose for a contradiction that the claim is false. Then there is some $\alpha < \mu^+$ such that for all $\alpha < i < \mu^+$ there is $j_i < \lambda_T$ and $B_i \subseteq A_{j_i}$ of cardinality at most $\lambda_T$ such that $\tp_\phi(a_i / A_{j_i})$ does not $(\chi, \phi)$-split over $B_i$. As $\mu^+ > \lambda_T$, by the pigeonhole principle, we can find some $I \subseteq \mu^+$ with $|I| = \mu^+$ such that $j_i = j_{i'}$ for all $i,i' \in I$. Write $j$ for $j_i$, where $i \in I$. As $\mu^+ > \mu = \mu^{\lambda_T} \geq |A_j|^{\lambda_T}$ we can apply the pigeonhole principle again to find $I' \subseteq I$ with $|I'| = \mu^+$ and $B_i = B_{i'}$ for all $i,i' \in I'$. Write $B$ for $B_i$, where $i \in I'$. We have that $A \subseteq A_j$, so for any distinct $i, i' \in I'$ we have that $\tp_\phi(a_i / A_j) \neq \tp_\phi(a_{i'} / A_j)$. We thus find $\mu^+ > 2^{2^{\lambda_T}}$ many types that do not $(\chi, \phi)$-split over a set of cardinality at most $\lambda_T$. This contradicts \thref{lem:non-splitting-facts}(iii) and completes the proof of the claim.

Using the claim we find some $i < \mu^+$ such that $a_i \not \in \bigcup_{j < \lambda_T} A_j$, because $\left| \bigcup_{j < \lambda_T} A_j \right| \leq \mu$. So for all $j < \lambda_T$ $\tp_\phi(a_i / A_j)$ $(\chi, \phi)$-splits over every $B \subseteq A_j$ of cardinality at most $\lambda_T$. By induction on $j < \lambda_T$ we define $b_j, b'_j, c_j \in A_{2j+2}$, such that:
\begin{enumerate}[label=(B\arabic*)]
\item writing $B_j = \{b_k, b'_k, c_k : k < j\}$, we have $B_j \subseteq A_{2j}$;
\item $\tp_\chi(b_j/B_j) = \tp_\chi(b'_j/B_j)$,
\item $\models \phi(a_i, b_j)$ and $\not \models \phi(a_i, b'_j)$,
\item $c_j \in A_{2j+1}$ is such that $\tp(c_j / B_j b_j b'_j) = \tp(a_i/B_j b_j b'_j)$.
\end{enumerate}

Let $j < \lambda_T$ and assume we have constructed $b_k, b'_k, c_k$ for all $k < j$. As $B_j \subseteq A_{2j}$ has cardinality at most $\lambda_T$, we have that $\tp_\phi(a_i/A_{2j})$ $(\chi, \phi)$-splits over $B_j$. We can thus find $b_j, b'_j \in A_{2j}$ such that $\tp_\chi(b_j / B_j) = \tp_\chi(b'_j / B_j)$ while $\phi(x, b_j) \in \tp(a_i/A_{2j})$ and $\phi(x, b'_j) \not \in \tp(a_i/A_{2j})$. By construction of $A_{2j+1}$, in particular by (A2), we can find $c_j \in A_{2j+1}$ realising $\tp(a_i/B_j b_j b'_j)$.

Let $(d_j d'_j e_j)_{j < \omega}$ be an indiscernible sequence based on $(b_j b'_j c_j)_{j < \lambda_T}$. We note the following two properties.
\begin{enumerate}[label=(C\arabic*)]
\item By (B2) we have for all $k < j < \lambda_T$ that $\models \chi(b_j, c_k)$ if and only if $\models \chi(b'_j, c_k)$, that is $\models \phi(c_k, b_j)$ if and only if $\models \phi(c_k, b'_j)$. So we must have $\models \phi(e_k, d_j)$ if and only if $\models \phi(e_k, d'_j)$ for all $k < j < \omega$.
\item By (B3) and (B4) we have for all $k \leq j < \lambda_T$ that $\models \phi(c_j, b_k)$ and $\not \models \phi(c_j, b'_k)$. So we must have $\models \phi(e_j, d_k)$ and $\not \models \phi(e_j, d'_k)$ for all $k \leq j < \omega$.
\end{enumerate}
Based on (C1) we distinguish two cases, and show that in each case $\phi(x, y)$ has the order property.
\begin{enumerate}
\item The case where for all $k < j < \omega$ we have $\models \phi(e_k, d'_j)$. By (C2) we have $\not \models \phi(e_j, d'_k)$ for all $k \leq j < \omega$. And we conclude by applying \thref{prop:order-property-indiscernible-sequence} to $(e_n, d'_n)_{n < \omega}$.
\item The case where for all $k < j < \omega$ we have $\not \models \phi(e_k, d_j)$. By (C2) we have that $\models \phi(e_j, d_k)$ for all $k \leq j < \omega$. Write $\omega^\op$ for $\omega$ with the opposite order, then $(e_n, d_{n+1})_{n \in \omega^\op}$ is an indiscernible sequence such that $\models \phi (e_k, d_j) \Leftrightarrow k <^\op j$. Applying compactness and an analogue of \thref{prop:order-property-indiscernible-sequence} we conclude that $\phi(x,y)$ has the order property.
\end{enumerate}
\end{proof}
\begin{remark}
\thlabel{rem:belkasmi-thesis}
We compare the work in this section to \cite[Chapter 4]{belkasmi_contributions_2012}. Their Definition 4.10 is a definition for the order property for formulas $\phi(x, y)$ where $x$ and $y$ are tuples of variables of the same length (and sorts). One quickly verifies that their order property implies our \thref{def:order-property}. Conversely, given $\phi(x, y)$ satisfying our \thref{def:order-property}, as witnessed by $(a_i)_{i < \omega}$, $(b_i)_{i < \omega}$ and $\psi(x, y)$, the formula $\theta(x_1 y_1, x_2 y_2) := \phi(x_1, y_2)$ has the order property in the sense of \cite{belkasmi_contributions_2012} as witnessed by $(a_i b_i)_{i < \omega}$ and negation $\psi'(x_1 y_1, x_2 y_2) := \psi(x_1, y_2)$. Another difference is that \cite{belkasmi_contributions_2012} treats bounded theories separately, proving in \cite[Lemme 4.8]{belkasmi_contributions_2012} that they are stable. However, as we noted in \thref{rem:boundedness}, there is no need for such special treatment: we have seen how stability of bounded theories fits in our approach in \thref{ex:bounded-stable}. Given the translation of the notion of stability for formulas, we get the same results as \cite{belkasmi_contributions_2012} on the level of theories. However, our version allows for local stability and comparison to further combinatorial properties on the level of formulas (e.g.\ \thref{2-tp-implies-op}).
\end{remark}

\section{Definitions of the combinatorial properties}
\label{sec:definitions}
In this section we gather the definitions of the combinatorial properties we will consider. The definitions are very similar to those we know from full first-order logic, and they do indeed coincide when considering a full first-order theory as a positive theory (\thref{rem:morleyisation}). The main ingredient, which we already used in \thref{def:order-property} for \OP, is the idea of \cite[Section 6]{haykazyan_existentially_2021} to introduce ``inconsistency witnesses''. Whenever a traditional definition would say that a set of formulas is inconsistent, we now require the satisfaction of a positive formula that implies the inconsistency of that set of formulas. For example, if we would normally say that $\{\phi(x, a_1), \phi(x, a_2)\}$ is inconsistent, we now want $\models \psi(a_1, a_2)$ where $\psi(y_1, y_2)$ is a negation of $\exists x(\phi(x, y_1) \wedge \phi(x, y_2))$. The importance of this is that we can then use compactness to change the size or shape of the set of parameters involved. For example, instead of only considering sequences of shape $\omega$ for \OP we can consider any infinite sequence.
\begin{definition}
A formula $\phi(x,y)$ has the \emph{independence property} (\IP) if there are $(a_i)_{i < \omega}$, $(c_\sigma)_{\sigma\in 2^\omega}$ and a negation $\psi(x,y)$ of $\phi(x,y)$ such that for all $i < \omega$ and $\sigma\in 2^\omega$ we have: 
\begin{align*}
&\models \phi(a_i, c_\sigma) \quad\text{if} \ \sigma(i) = 1, \\
&\models \psi(a_i, c_\sigma) \quad\text{if} \ \sigma(i) = 0.
\end{align*}
\end{definition}

The study of \IP in positive logic has been initiated in the recent preprint \cite{dobrowolski_amalgamation_2023} by Dobrowolski and Mennuni.

\begin{definition}
\thlabel{def:tree-language}
Let $\kappa$ and $\lambda$ be (potentially finite) cardinals. As usual, we will consider the set $\kappa^{<\lambda}$ of functions $\eta: \alpha \to \kappa$ where $\alpha < \lambda$, as a tree. The partial order on the tree is given by $\eta \preceq \mu$ if $\mu$ extends $\eta$ as a function. We call $\eta$ and $\mu$ \emph{incomparable} if $\eta \not \preceq \mu$ and $\mu \not \preceq \eta$. For any $\eta, \mu \in \kappa^{<\lambda}$ we write $\eta^\frown \mu$ for their concatenation (viewing the functions as strings of ordinals $< \kappa$).
\end{definition}
\begin{definition}
For a natural number $k \geq 2$, a formula $\phi(x,y)$ has the \emph{k-tree property} ($k$-\TP) if there are $(a_\eta)_{\eta \in \omega^{< \omega}}$ and a negation $\psi(y_1, \ldots, y_k)$ of the formula $\exists x (\phi(x, y_1) \wedge \ldots \wedge \phi(x, y_k))$ such that:
\begin{itemize}
    \item[(1)] for all $\sigma \in \omega^\omega$ the set $\{ \phi(x, a_{\sigma|_n}): n < \omega\}$ is consistent,
    \item[(2)] for all $\eta \in \omega^{< \omega}$ and $i_1 < \ldots < i_k < \omega$ we have $\models \psi(a_{\eta^\frown i_1}, \ldots, a_{\eta^\frown i_k})$.
\end{itemize}

A formula $\phi(x,y)$ has the \emph{tree property} (\TP) if there exists a natural number $k \geq 2$ such that $\phi(x,y)$ has $k$-\TP.
\end{definition}

\begin{definition}
A formula $\phi(x,y)$ has the \emph{tree property of the first kind} (\TP[1]) if there are $(a_\eta)_{\eta \in \omega^{< \omega}}$ and a negation $\psi(y_1,y_2)$ of $\exists x (\phi(x,y_1) \wedge \phi(x,y_2))$ such that:
\begin{itemize}
    \item[(1)] for all $\sigma \in \omega^\omega$ the set $\{\phi(x,a_{\sigma |_n}) : n < \omega \}$ is consistent,
    \item[(2)] for all incomparable $\mu,\eta \in \omega^{<\omega}$ we have $\models \psi(a_\mu,a_\eta)$.
\end{itemize}
\end{definition}

\begin{definition}
A formula $\phi(x,y)$ has the \emph{$k$-tree property of the second kind} ($k$-\TP[2]) if there are $(a_{i,j})_{i,j < \omega}$ and a negation $\psi(y_1, \ldots, y_k)$ of the formula $\exists x (\phi(x, y_1) \wedge \ldots \wedge \phi(x, y_k))$ such that:
\begin{itemize}
    \item[(1)] for all $\sigma \in \omega^\omega$ the set $\{ \phi(x, a_{i, \sigma(i)}): i < \omega\}$ is consistent,
    \item[(2)] for all $i < \omega$ and $j_1 < \ldots < j_k < \omega$ we have $\models \psi(a_{i,j_1}, \ldots, a_{i,j_k})$.
\end{itemize}
A formula $\phi(x,y)$ has the \emph{tree property of the second kind} (\TP[2]) if there exists a natural number $k \geq 2$ such that $\phi(x,y)$ has $k$-\TP[2].
\end{definition}

The definition of \TP[2] in positive logic first appeared in \cite{haykazyan_existentially_2021}, as did the following definition of \SOP[1].

\begin{definition}
A formula $\phi(x,y)$ has the \emph{$1$-strong order property} (\SOP[1]) if there are $(a_\eta)_{\eta \in 2^{<\omega}}$ and a negation $\psi(y_1,y_2)$ of $\exists x (\phi(x,y_1) \wedge \phi(x,y_2))$ such that:
\begin{itemize}
    \item[(1)] for all $\sigma \in 2^\omega$ the set $\{\phi(x,a_{\sigma |_n}) : n < \omega \}$ is consistent,
    \item[(2)] for all $\mu, \eta \in 2^{<\omega}$, if $\mu^\frown 0 \preceq \eta$ then $\models \psi(a_{\mu^\frown 1}, a_\eta)$.
\end{itemize}
\end{definition}

\begin{definition}
A formula $\phi(x,y)$ has the \emph{$2$-strong order property} (\SOP[2]) if there are $(a_\eta)_{\eta \in 2^{< \omega}}$ and a negation $\psi(y_1,y_2)$ of $\exists x (\phi(x,y_1) \wedge \phi(x,y_2))$ such that:
\begin{itemize}
    \item[(1)] for all $\sigma \in 2^\omega$ the set $\{\phi(x,a_{\sigma |_n}) : n < \omega \}$ is consistent,
    \item[(2)] for all incomparable $\mu,\eta \in 2^{<\omega}$ we have $\models \psi(a_\mu,a_\eta)$.
\end{itemize}
\end{definition}

\begin{definition}
A theory $T$ has one of the properties above (\OP, \IP, $k$-\TP, $k$-\TP[2], \SOP[1], \SOP[2]) if there exists a formula witnessing it. 
\end{definition}

\begin{definition}
A theory $T$ has the \emph{$3$-strong order property} (\SOP[3]) if there are formulas $\phi_0(x,y)$ and $\phi_1(x,y)$, a sequence $(a_i)_{i < \omega}$, and a negation $\psi(y_1,y_2)$ of $\exists x (\phi_0(x,y_2) \wedge \phi_1(x,y_1))$ such that:
\begin{itemize}
    \item[(1)] for all $k < \omega$ the $\{\phi_0(x,a_i):i < k\} \cup \{\phi_1(x,a_j):j \geq k\} $ is consistent,
    \item[(2)] for all $i<j<\omega$ we have $\models \psi(a_i,a_j)$.
\end{itemize}
\end{definition}

In the full first-order setting, \SOP[3] is usually defined on the level of formulas with a definition which easily generalizes to any natural number larger than $3$, giving rise to the notion of an \SOP[n] formula (or theory). This definition heavily relies on the use of negation, which forms an obstruction to translating it to the positive setting, see \thref{rem:sop}. The definition given here, on the level of theories, is based on \cite[Claim 2.19]{shelah_toward_1996}.

\begin{definition}
If a theory $T$ does not have one of the properties \OP, \IP, \TP, \TP[1], \TP[2], \SOP[1], \SOP[2], \SOP[3], we say that $T$ is \NOP, \NIP, \NTP, \NTP[1], \NTP[2], \NSOP[1], \NSOP[2], \NSOP[3] respectively.
\end{definition}

\section{Implications between the combinatorial properties}
\label{sec:implication-between-the-properties}
In this section we prove \thref{implication-diagram} by proving the implications between the various properties of positive theories defined in Section \ref{sec:definitions}. We break up the proof in its individual components, stating each arrow separately. We start from the left-most implication and make our way inside the diagram. Some of the implications will be proved on a formula level (e.g.\ \thref{proposition:sop2ifftp1}) and for some implications this will only happen on a theory level (e.g.\ \thref{cor:tp-iff-2-tp}).
\begin{remark}
\thlabel{rem:converse-implications}
We make some remarks about the strictness of the implications in \thref{implication-diagram}.
\begin{itemize}
\item The strictness of the implication \SOP[3] $\implies$ \SOP[2] is an open question even in full first-order logic. See also \thref{q:sop3-iff-sop2}.
\item Mutchnik's recent preprint \cite{mutchnik_nsop_2_2022} proves the implication \SOP[1] $\implies$ \SOP[2] in full first-order logic. As the machinery used there is considerably more involved than what we apply here, we do not deal with this problem and leave it to future work, see \thref{q:sop1-iff-sop2}.
\item In full first-order logic we have $k$-\TP[2] $\implies$ $2$-\TP[2]. A recent preprint by the third author proves this for thick theories \cite[Theorem 1.4]{kamsma_positive_2023}. See also \thref{q:2-tp2-iff-k-tp2}.
\item The remaining implications are known to be strict, already in full first-order logic.
\end{itemize}
\end{remark}

\begin{proposition}
If a theory $T$ has \SOP[3] then it has \SOP[2].
\end{proposition}
This proof is based on the similar argument in Proposition 1.8 of \cite{conant_dividing_2012}, adapted to the definition of \SOP[3] at the level of theories that we use here.
\begin{proof}
Assume $T$ has \SOP[3], witnessed by formulas $\phi_0(x, y)$, $\phi_1(x, y)$ and a negation $\psi(y_1, y_2)$ of $\exists x(\phi_0(x, y_2) \wedge \phi_1(x, y_1))$. By compactness we find a sequence $(b_q)_{q \in \Q}$ such that:
\begin{enumerate}[label=(\arabic*)]
    \item for all $t \in \Q$ the set $\{ \phi_0(x,b_q) : q < t \} \cup \{\phi_1(x,b_r):r \geq t\}$ is consistent,
    \item for all $q < r$ in $\Q$ we have $\models \psi(b_q,b_r)$.
\end{enumerate}
Consider $\chi(x, y_1, y_2) := \phi_0(x, y_1) \wedge \phi_1(x, y_2)$. We inductively define a tree, indexed by $2^{<\omega}$, which will witness that $\chi$, and hence $T$, has \SOP[2].
\begin{align*}
c_\varnothing &= (b_0,b_1) , \\
c_{\eta^\frown 0} &= (b_q,b_{\frac{2}{3}q + \frac{1}{3}r}) \quad\text{for } c_\eta=(b_q,b_r), \\
c_{\eta^\frown 1} &= (b_{\frac{1}{3}q + \frac{2}{3}r}, b_r) \quad\text{for } c_\eta=(b_q,b_r).
\end{align*}
Let moreover $\theta(y_1,y_2,y_3,y_4)$ denote the formula $\psi(y_4,y_1) \lor \psi(y_2, y_3)$. 

We claim that $\theta(y_1,y_2,y_3,y_4)$ is a negation of $\exists x (\chi(x,y_1,y_2) \wedge \chi(x,y_3,y_4))$. Indeed, assume that $\theta(a,b,c,d)$ holds. Then either $\psi(d,a)$ holds or $\psi(b,c)$ holds. By definition of $\psi$ we have that in the first case $\exists x (\phi_0(x,a) \wedge \phi_1(x,d))$ does not hold, and in the second case $\exists x (\phi_0(x, c) \wedge \phi_1(x, b))$. Either way, we have that $\exists x(\phi_0(x,a) \wedge \phi_1(x,d) \wedge \phi_0(x, c) \wedge \phi_1(x, b))$ does not hold. The claim now follows from the definition of $\chi$.

We will now verify that $(c_\eta)_{\eta \in 2^{<\omega}}$ and $\theta$ witness that $\chi$ has \SOP[2]. For consistency along the branches, let $\sigma \in 2^\omega$ and $n < \omega$. Then there are $0=q_0\leq\dots\leq q_n<r_n\leq\dots\leq r_0=1$ such that for $0 \leq i \leq n$, $c_{\sigma |_i}=(b_{q_i},b_{r_i})$. Taking $t = r_n$ in (1) above, we see that
\[
\{\phi_0(x,b_{q_i}) : i \leq n\} \cup \{\phi_1(x,b_{r_i}): i \leq n\}
\]
is consistent. Thus $\{ \chi(x,c_{\sigma|_i}) : i < \omega \}$ is finitely consistent, and hence consistent.

Now let $\mu, \eta$ be incomparable. Then there are $q<r<s<t$ such that either $c_\mu=(b_q,b_r)$ and $c_\eta=(b_s,b_t)$, or $c_\mu=(b_s,b_t)$ and $c_\eta=(b_q,b_r)$. In both cases, since $r < s$, we have $\models \psi(b_r,b_s)$. That means both $\models \theta(b_q,b_r,b_s,b_t)$ and $\models \theta(b_s,b_t,b_q,b_r)$, giving $\models \theta(c_\mu, c_\eta)$ in any case. This concludes the proof.
\end{proof}

\begin{proposition}\thlabel{proposition:sop2ifftp1}
A formula $\phi(x,y)$ has \SOP[2] if and only if it has \TP[1].
\end{proposition}

\begin{proof}
One direction is obvious: if $(a_\eta)_{\eta \in \omega^{<\omega}}$ and $\psi(y_1,y_2)$ witness \TP[1] of $\phi$, then $(a_\eta)_{\eta \in 2^{<\omega}}$ and $\psi(y_1,y_2)$ witness \SOP[2] of $\phi$.

For the converse, let us then assume that $\phi(x,y)$ has \SOP[2], witnessed by $(a_\eta)_{\eta \in 2^{<\omega}}$ and $\psi(y_1,y_2)$. We inductively define a function $h:\omega^{<\omega} \rightarrow 2^{< \omega}$ as 
\begin{align*}
&h(\varnothing)=\varnothing , \\
&h(\eta^\frown i)=h(\eta)^\frown (0)^{i \frown} 1 \quad\text{for } i<\omega.
\end{align*}
Note that $\eta \preceq \mu$ implies $h(\eta) \preceq h(\mu)$ and so for any $\sigma \in \omega^\omega$ there is $\sigma' \in 2^\omega$ such that $\{h(\sigma|_n) : n < \omega\} \subseteq \{\sigma'|_n : n < \omega\}$.

Define a tree $(b_\eta)_{\eta \in \omega^{<\omega}}$ by $b_\eta = a_{h(\eta)}$. We verify that this tree witnesses \TP[1] for $\phi(x, y)$,  with the same negation $\psi(y_1, y_2)$.

For any $\sigma \in \omega^\omega$ there is $\sigma' \in 2^\omega$ such that
\[
\{ \phi(x, b_{\sigma|_n}) : n < \omega \} =
\{ \phi(x, a_{h(\sigma|_n)}) : n < \omega \} \subseteq
\{ \phi(x, a_{\sigma'|_n}) : n < \omega \}.
\]
The rightmost set is consistent because $(a_\eta)_{\eta \in 2^{<\omega}}$ witnesses \SOP[2] for $\phi(x, y)$, so the leftmost set is consistent.

Let $\mu, \eta \in \omega^{< \omega}$ be incomparable. Then there are $\gamma, \mu_0, \eta_0 \in \omega^{<\omega}$ and $i\neq j < \omega$ such that $\mu=\gamma^\frown i^\frown \mu_0$ and $\eta=\gamma^\frown j ^\frown \eta_0$. By definition of $h$, there are $\mu_1,\eta_1 \in 2^{<\omega}$ such that $h(\eta)=h(\gamma)^\frown (0)^{i \frown} 1 ^\frown \eta_1$ and $h(\mu)=h(\gamma)^\frown (0)^{j \frown} 1 ^\frown \mu_1$. If $i<j$ then $h(\eta)$ has a $1$ in a place where $h(\mu)$ has a $0$, and thus they are incomparable. Similarly, if $j < i$ then $h(\eta)$ and $h(\mu)$ are again incomparable. Hence, by definition of \SOP[2], $\models \psi(a_{h(\eta)}, a_{h(\mu)})$, and so $(b_\eta)_{\eta \in \omega^{<\omega}}$ and $\psi$ witness \TP[1] of $\phi$.
\end{proof}

\begin{proposition}
If a formula $\phi(x,y)$ has \SOP[2] then it has \SOP[1].
\end{proposition}
\begin{proof}
Suppose $\phi(x,y)$ has \SOP[2], witnessed by $(a_\eta)_{\eta \in 2^{<\omega}}$ and $\psi(y_1,y_2)$. Let $\mu, \eta \in 2^{<\omega}$, and $\mu^\frown 0 \preceq \eta$. Then $\mu^\frown 1$ and $\eta$ are incomparable, so by \SOP[2] we have $\models \psi(a_{\mu^\frown 1}, a_\eta)$. Hence the second clause in the definition of \SOP[1] is satisfied. The first clause is the same as the first clause of the definition of \SOP[2], and therefore $\phi(x,y)$ has \SOP[1], witnessed again by $(a_\eta)_{\eta \in \omega^{<\omega}}$ and $\psi(y_1,y_2)$.
\end{proof}

\begin{proposition}
If a formula $\phi(x,y)$ has \SOP[1] then it has $2$-\TP.
\end{proposition}
\begin{proof}
Suppose that $\phi(x,y)$ has \SOP[1], witnessed by $(a_\eta)_{\eta \in 2^{<\omega}}$ and $\psi(y_1,y_2)$. We define $h:\omega^{<\omega} \rightarrow 2^{< \omega}$ as in the proof of Proposition \ref{proposition:sop2ifftp1}. Again, we define $(b_\eta)_{\eta \in \omega^{<\omega}}$ by $b_\eta = a_{h(\eta)}$, so that for any $\sigma \in \omega^\omega$, we get that $\{ \phi(x, b_{\sigma|_n} : n < \omega \}$ is consistent.

Now let $\eta \in \omega^{<\omega}$ and let $i < j < \omega$. Then
\[
h(\eta)^\frown (0)^{i+1} \preceq h(\eta)^\frown (0)^{j \frown} 1=h(\eta ^\frown j).
\]
We have that $h(\eta^\frown i)=h(\eta)^\frown (0)^{i \frown}(1)$ and hence, by the second clause in the definition of \SOP[1] we have $\models \psi(a_{h(\eta ^\frown i)}, a_{h(\eta^ \frown j)})$. Hence $(b_\eta)_{\eta \in \omega^{< \omega}}$ and $\psi$ witness $2$-\TP for $\phi$.
\end{proof}
The following argument is based on \cite[Theorem III.7.7]{shelah_classification_1990}.
\begin{theorem}
\thlabel{thm:tp-implies-2-tp}
If $\phi(x, y)$ has \TP then for some $k'$ the conjunction $\phi(x, y_1) \wedge \ldots \wedge \phi(x, y_{k'})$ has $2$-\TP.
\end{theorem}
\begin{proof}
Let $\kappa=|T|^+$. We will find a set of parameters $A$ and a set $B \subseteq A^\kappa$ such that:
\begin{enumerate}[label=(\roman*)]
\item $|B| > |A|^{< \kappa} + 2^{2^\kappa}$;
\item if $B' \subseteq B$ and $|B'| > 2^\kappa$ then $\{ \phi(x, b(\alpha)) : b \in B', \alpha < \kappa \}$ is inconsistent; 
\item for any $b \in B$ we have that $\{ \phi(x, b(\alpha)) : \alpha < \kappa \}$ is consistent.
\end{enumerate}
Let $\lambda = \beth_\kappa(|T| + 2^\kappa)$. By compactness we find a tree $(a_\eta)_{\eta \in \lambda^{< \kappa}}$ such that for every $\sigma \in \lambda^\kappa$ the set $\{ \phi(x, a_{\sigma|_\alpha}) : \alpha < \kappa \}$ is consistent, but there exists $2 \leq k < \omega$ such that for any $\eta \in \lambda^{<\kappa}$ the set $\{ \phi(x, a_{\eta^\frown i}) : i < \lambda \}$ is $k$-inconsistent. Because of this last property we may assume that for any $\eta \in \lambda^{<\kappa}$ all the terms of $(a_{\eta^\frown i})_{i < \lambda}$ are distinct.

Write $A = \{a_\eta : \eta \in \lambda^{< \kappa}\}$ and for $\sigma \in \lambda^\kappa$ we define $b_\sigma: \kappa \to A$ by $b_\sigma(\alpha) = a_{\sigma|_\alpha}$. We claim that this $A$ together with $B = \{ b_\sigma : \sigma \in \lambda^\kappa \}$ satisfies (i)--(iii).
\begin{enumerate}[label=(\roman*)]
\item As $\cf(\lambda) = \kappa$ we have $|B| = \lambda^\kappa > \lambda = \lambda^{<\kappa} = |A|^{< \kappa}$, and the required inequality follows.  Here we use that for distinct $\sigma, \sigma' \in \lambda^\kappa$ we have $b_\sigma \neq b_{\sigma'}$, which follows from our earlier assumption that the terms of $(a_{\eta^\frown i})_{i < \lambda}$ are distinct for any $\eta \in \lambda^{<\kappa}$.
\item Let $B' \subseteq B$ be such that $\{ \phi(x, b(\alpha)) : b \in B', \alpha < \kappa \}$ is consistent, we will show that $|B'| \leq 2^\kappa$. Define $X = \{ \sigma \in \lambda^\kappa : b_\sigma \in B' \}$, so $\{ \phi(x, a_{\sigma|_\alpha}) : \sigma \in X, \alpha < \kappa \}$ is consistent. By construction of $(a_\eta)_{\eta \in \lambda^{< \kappa}}$ we then must have for all $\eta \in \lambda^{<\kappa}$ that the branches in $X$ pass through at most $k-1$ immediate successors of $\eta$, that is:
\[
|\{ i < \lambda : \text{there is } \sigma \in X \text{ such that } \eta^\frown i \preceq \sigma \}| < k.
\]
After re-indexing we then have that $X \subseteq k^\kappa$ and hence $|B'| = |X| \leq 2^\kappa$.
\item This is just consistency of $\{ \phi(x, a_{\sigma|_\alpha}) : \alpha < \kappa \}$ for every $\sigma \in \lambda^\kappa$.
\end{enumerate}
With $A$ and $B$ be as above, let $\mu = |A|^{< \kappa} + 2^{2^\kappa}$. We will find a cardinal $\kappa \leq \chi \leq 2^\kappa$ and a set $S \subseteq A^\chi$ such that:
\begin{enumerate}[label=(\arabic*)]
\item $|S| = \mu^+$;
\item for any distinct $s, s' \in S$ we have that $\{\phi(x, s(\alpha)) : \alpha < \chi\} \cup \{\phi(x, s'(\alpha)) : \alpha < \chi\}$ is inconsistent;
\item for any $s \in S$ the set $\{\phi(x, s(\alpha)) : \alpha < \chi\}$ is consistent;
\item for any $s, s' \in S$, viewing them as infinite tuples, we have $s \equiv s'$.
\end{enumerate}
First we may assume $|B| = \mu^+$. We inductively construct $U_i \subseteq B$ as follows: $U_i$ is a maximal subset such that $U_j \cap U_i = \emptyset$ for all $j < i$ and $\{\phi(x, b(\alpha)) : b \in U_i, \alpha < \kappa\}$ is consistent. Note that the latter implies that $|U_i| \leq 2^\kappa$ by (ii), which together with (iii) allows us to continue the construction until we have constructed $\{U_i\}_{i < \mu^+}$. By the pigeonhole principle we may assume that all the $U_i$ have the same cardinality. For all $i <\mu^+$, let $A_i=\{b(\alpha) :b \in U_i, \alpha <\kappa \}$, $\chi=|A_i|=\kappa \cdot |U_i|$, and $s_i \in A^\chi$ an enumeration of $A_i$. If we let $S = \{s_i : i < \mu^+\}$ then it satisfies (1)--(3), and by the pigeonhole principle we can replace $S$ by a subset to also ensure (4).

For $s \in S$ we now define pairs $(v^s_\alpha, \psi^s_\alpha)$ inductively on $\alpha < \delta_s$, where $\delta_s$ is the first $\alpha$ for which $(v^s_\alpha, \psi^s_\alpha)$ cannot be defined. We require:
\begin{enumerate}[label=(\Alph*)]
\item $v^s_\alpha \subseteq \chi$ is finite;
\item there is $\{s_n\}_{n < \omega} \subseteq S$ with $s_n(j) = s(j)$ for all $n < \omega$ and all $j \in \bigcup_{\beta < \alpha} v^s_\beta$;
\item $\psi^s_\alpha((y_\gamma)_{\gamma \in v^s_\alpha}, (y'_\gamma)_{\gamma \in v^s_\alpha})$ is a negation of $\exists x \left( \bigwedge_{\gamma \in v^s_\alpha} \phi(x, y_\gamma) \wedge \phi(x, y'_\gamma) \right)$;
\item for any distinct $n, m < \omega$ we have $\models \psi^s_\alpha((s_n(\gamma))_{\gamma \in v^s_\alpha}, (s_m(\gamma))_{\gamma \in v^s_\alpha})$.
\end{enumerate}
We will show that there is $s \in S$ such that $\delta_s \geq \kappa$. Suppose for a contradiction that $\delta_s < \kappa$ for all $s \in S$. There are $(\chi^{< \omega} \cdot |T|)^{< \kappa} = \chi^{<\kappa} \leq 2^\kappa \leq \mu$ many possible sequences $(v^s_\alpha, \psi^s_\alpha)_{\alpha < \delta_s}$. So by the pigeonhole principle there is $S_1 \subseteq S$ with $|S_1| = \mu^+$ and for all $s,s' \in S_1$ we have $\delta_s = \delta_{s'}$ and $(v^s_\alpha, \psi^s_\alpha)_{\alpha < \delta_s} = (v^{s'}_\alpha, \psi^{s'}_\alpha)_{\alpha < \delta_{s'}}$. Write $(v_\alpha, \psi_\alpha)_{\alpha < \delta} = (v^s_\alpha, \psi^s_\alpha)_{\alpha < \delta_s}$ for some $s \in S_1$. As $\delta < \kappa$ we have that $|\bigcup_{\beta < \delta} v_\beta| < \kappa$. So as $|A|^{<\kappa} \leq \mu$ we can again apply the pigeonhole principle to find $S_2 \subseteq S_1$ with $|S_2| = \mu^+$ and for any $s, s' \in S_2$ we have that $s(j) = s'(j)$ for all $j \in \bigcup_{\beta < \delta} v_\beta$.

By (2) we have that any two distinct $s, s' \in S_2$ the set $\{\phi(x, s(\alpha)) : \alpha < \chi\} \cup \{\phi(x, s'(\alpha)) : \alpha < \chi\}$ is inconsistent. So we can assign a finite $u_{s,s'} \subseteq \chi$ and $\theta_{s,s'}((y_\gamma)_{\gamma \in u_{s,s'}}, (y'_\gamma)_{\gamma \in u_{s,s'}})$ to each such a pair, such that $\theta_{s,s'}((y_\gamma)_{\gamma \in u_{s,s'}}, (y'_\gamma)_{\gamma \in u_{s,s'}})$ is a negation of $\exists x \left( \bigwedge_{\gamma \in u_{s,s'}} \phi(x, y_\gamma) \wedge \phi(x, y'_\gamma) \right)$ and $\models \theta_{s,s'}((s(\gamma))_{\gamma \in u_{s,s'}}, (s'(\gamma))_{\gamma \in u_{s,s'}})$. This defines a colouring function on $[S_2]^2$ with $(\chi^{<\omega} \cdot |T|) = \chi$ many colours. As $\mu^+ \geq (2^\chi)^+$ we can apply the Erd\H{o}s-Rado theorem to find $S_3 \subseteq S_2$ with $|S_3| = \chi^+$ such that $u = u_{s,s'}$ and $\theta = \theta_{s,s'}$ do not depend on the pair $s,s' \in S_3$. However, for any $s \in S_3$ we could now have taken $(v^s_{\delta_s}, \psi^s_{\delta_s})$ to be $(u, \theta)$, contradicting the definition of $\delta_s$. Indeed (A) and (C) follow immediately from the construction of $u$ and $\theta$. For (B) and (D) any $\{s_n\}_{n < \omega} \subseteq S_3$ suffices, which exists because $|S_3| = \chi^+$ is infinite, then (B) follows because this is also a subset of $S_2$ and (D) follows from the construction of $\theta$.

There is thus some $s \in S$ such that $\delta_s \geq \kappa$. As $\kappa = |T|^+$ there is some $k'$ and $\psi(\bar{y}, \bar{y}') := \psi(y_1, \ldots, y_{k'}, y'_1, \ldots, y'_{k'})$ such that there are infinitely many $\alpha < \kappa$ with $|v^s_\alpha| = k'$ and $\psi^s_\alpha = \psi$ (after renaming variables). For convenience we may as well assume that this happens for all $\alpha < \omega$. We will show that $\phi(x, y_1) \wedge \ldots \wedge \phi(x, y_{k'})$ has $2$-\TP. The relevant negation will be $\psi$, so we need to construct the tree $(c_\eta)_{\eta \in \omega^{<\omega}}$ of parameters. For $\alpha < \omega$ we write $\bar{s}(\alpha)$ for the tuple $(s(\beta))_{\beta \in v^s_\alpha}$. We now construct $(c_\eta)_{\eta \in \omega^{<\omega}}$ by induction on the length of $\eta \in \omega^{<\omega}$, such that for $\eta \in \omega^n$ we have that $(c_{\eta|_\alpha})_{\alpha \leq n} \equiv (\bar{s}(\alpha))_{\alpha \leq n}$.

We can simply take $c_\emptyset = \bar{s}(0)$. Now assume we have constructed $c_{\eta}$ for $\eta \in \omega^n$, we will construct $c_{\eta^\frown i}$ for all $i < \omega$. By an automorphism we may assume $(c_{\eta|_\alpha})_{\alpha \leq n} = (\bar{s}(\alpha))_{\alpha \leq n}$. Let $(s_i)_{i < \omega}$ be as in (B) for $v^s_{n+1}$. We set $c_{\eta^\frown i} = \bar{s}_i(n+1)$ for all $i < \omega$. Then we get $$c_{\eta^\frown i} (c_{\eta|_\alpha})_{\alpha \leq n} = \bar{s}_i(n+1) (\bar{s}(\alpha))_{\alpha \leq n} = \bar{s}_i(n+1) (\bar{s}_i(\alpha))_{\alpha \leq n} \equiv \bar{s}(n+1) (\bar{s}(\alpha))_{\alpha \leq n}.$$ Here the second equality follows from (B) and the third equivalence follows from (4).

%By (B) we then have that $c_{\eta^\frown i} (c_{\eta|_\alpha})_{\alpha \leq n} = \bar{s}_i(n+1) (\bar{s}(\alpha))_{\alpha \leq n} = \bar{s}_i(n+1) (\bar{s}_i(\alpha))_{\alpha \leq n}$, so the induction hypothesis follows from (4).

We are left to verify that the tree $(c_\eta)_{\eta \in \omega^{<\omega}}$ is indeed an instance of $2$-\TP. Indeed, for any $\eta \in \omega^{<\omega}$ and $i < j < \omega$ we have $\models \psi(c_{\eta^\frown i}, c_{\eta^\frown j})$ by (D). Finally, for any $\sigma \in \omega^\omega$ we have by the induction hypothesis that $(c_{\sigma|_\alpha})_{\alpha < \omega} \equiv (\bar{s}(\alpha))_{\alpha < \omega}$, so the required consistency follows from (3).
\end{proof}
\begin{corollary}
\thlabel{cor:tp-iff-2-tp}
A theory $T$ has \TP if and only if it has $2$-\TP.
\end{corollary}

\begin{proposition}\thlabel{prop:tp-implies-many-types}
Suppose $\phi(x,y)$ has $2$-\TP. Then there exists an infinite set $B$ such that $|\S_\phi (B)| > (|B| + |T|)^{|T|}$.
\end{proposition}
\begin{proof}
Let $\kappa = \beth_{|T|^+}$, then $\omega^{<\kappa} = \kappa$ and $\kappa^{|T|} = \kappa$. To see the latter we note that for any $f: |T| \to \kappa$, there is $\alpha < |T|^+$ such that the image of $f$ is contained in $\beth_\alpha$. Hence $\kappa^{|T|} = \bigcup_{\alpha < |T|^+} \beth_\alpha^{|T|}$, and $\beth_\alpha^{|T|} \leq \beth_{\alpha+1}$, from which the equality follows.

We assume $\phi(x,y)$ has $2$-\TP, so by compactness we find $(b_\eta)_{\eta \in \omega^{< \kappa}}$ and a negation $\psi(y_1,y_2)$ of the formula $\exists x  (\phi(x, y_1) \wedge \phi(x, y_2))$ witnessing $2$-\TP.

Let $B = \{ b_\eta : \eta \in \omega^{< \kappa} \}$. For $\sigma \in \omega^\kappa$ let $a_\sigma$ be a realisation of $\{ \phi(x, b_{\sigma|_\alpha}): \alpha < \kappa \}$. Given distinct $\sigma_1, \sigma_2 \in \omega^\kappa$, we have $\tp_\phi(a_{\sigma_1}/B) \neq \tp_\phi(a_{\sigma_2}/B)$. Indeed, let $\eta \in \omega^{< \kappa}$ be such that $\eta \preceq \sigma_1, \sigma_2$ but there are $i \neq j < \omega$ such that $\eta^\frown i \preceq \sigma_1$ and $\eta^\frown j \preceq \sigma_2$. Without loss of generality, assume $i < j$. Then $\models \psi(b_{\eta^\frown i}, b_{\eta^\frown j})$ and so because we have $\models \phi(a_{\sigma_1}, b_{\eta^\frown i})$ and $\models \phi(a_{\sigma_2}, b_{\eta^\frown j})$ we cannot have $\models \phi(a_{\sigma_2}, b_{\eta^\frown i})$.

We thus find $\omega^\kappa > \kappa$ many types in $\S_\phi(B)$, while at the same time $(|B| + |T|)^{|T|} = (\omega^{<\kappa} + |T|)^{|T|} = \kappa^{|T|} = \kappa$ by our choice of $\kappa$.
\end{proof}

\begin{corollary}
\thlabel{2-tp-implies-op}
If a formula $\phi(x,y)$ has $2$-\TP then it has \OP.
\end{corollary}
\begin{proof}
By \thref{prop:tp-implies-many-types} and \thref{thm:stable-formula}.
\end{proof}

\begin{proposition}
If a formula $\phi(x,y)$ has $k$-\TP[2] then it has $k$-\TP.
\end{proposition}
\begin{proof}
Suppose that $\phi(x,y)$ has $k$-\TP[2], witnessed by $(a_{i,j})_{i,j < \omega}$ and $\psi(y_1,\ldots, y_k)$. We construct $(b_\eta)_{\eta\in\omega^{<\omega}}$ such that together with $\psi(y_1, \ldots, y_k)$ they witness $k$-\TP. For $\eta\in\omega^{<\omega}$, let $\ell(\eta)$ be the length (domain) of $\eta$ and let $t(\eta)$ be the last element of $\eta$ and $t(\emptyset) = 0$. Define $b_\eta = a_{\ell(\eta), t(\eta)}$.

For any $\sigma\in \omega^\omega$ we have that $\{ \phi(x, b_{\sigma|_n}) : n < \omega \} = \{ \phi(x, a_{n, t(\sigma|_n)}) : n < \omega\}$ is consistent. Let now $\eta \in \omega^{< \omega}$, and write $n = \ell(\eta) +1$. Then for any $i_1 < \ldots < i_k < \omega$ we have $\models \psi(a_{n, i_1}, \ldots, a_{n,i_k})$. This is the same as $\models \psi(b_{\eta^\frown i_1}, \ldots, b_{\eta^\frown i_k})$. Hence, $(b_\eta)_{\eta\in\omega^{<\omega}}$ and $\psi(y_1, \ldots, y_k)$ witness $\phi(x,y)$ having $k$-\TP.
\end{proof}

\begin{proposition}
If a theory $T$ has $2$-\TP[2] then it has \IP.
\end{proposition}
\begin{proof}
Assume $T$ has $2$-\TP[2] witnessed by the formula $\phi(x,y)$, $(a_{i,j})_{i, j < \omega}$ and a negation $\psi(y_1,y_2)$ of the formula $\exists x (\phi(x, y_1) \wedge \phi(x, y_2))$. Then for every $\sigma \in 2^\omega \subseteq \omega^\omega$, there exists $c_\sigma$ such that for all $i < \omega$ we have $\models \phi(c_\sigma, a_{i, \sigma(i)})$.

Consider the formulas $\chi(z_1 z_2, t) := \phi(t, z_2) $ and $\xi(z_1 z_2, t) := \phi(t, z_1) \land \psi(z_1, z_2) $. Also for $i<\omega$ let $b_i$ be the tuple $a_{i,0} \ a_{i,1}$. We are going to show that $\chi$ has \IP witnessed by $(b_i)_{i< \omega}$, $(c_\sigma)_{\sigma\in 2^\omega}$ and $\xi$.

First of all note that $\xi$ is indeed a negation of $\chi$, since
\[
T \models \neg \exists z_1 z_2 t \ (\phi(t, z_2) \land \phi(t, z_1) \land \psi(z_1, z_2)).
\]
Now take any $i < \omega$ and $\sigma \in 2^\omega$. If $\sigma(i) = 1$, then we have $\models \phi(c_\sigma, a_{i, 1})$ and therefore $\models \chi(b_i, c_\sigma)$. If $\sigma(i) = 0$, then we have $\models \phi(c_\sigma, a_{i, 0})$ as well as $\models \psi(a_{i,0}, a_{i, 1})$ and therefore $\models \xi(b_i, c_\sigma)$. Hence, $(b_i)_{i< \omega}$, $(c_\sigma)_{\sigma\in 2^\omega}$ and $\xi$ witness \IP of $\chi$ and $T$ has \IP.
\end{proof}

\begin{proposition}
If a formula $\phi(x,y)$ has \IP then it has \OP.
\end{proposition}
\begin{proof}
Suppose $\phi(x, y)$ has \IP, witnessed by $(a_i)_{i < \omega}$, $(c_\sigma)_{\sigma\in 2^\omega}$ and $\psi(x,y)$. We use the same $\psi(x,y)$ and $(a_i)_{i < \omega}$ to show that $\phi(x,y)$ has \OP. Let $\sigma_j \in 2^\omega$ be defined by 
\[
   \sigma_j(i) =
  \begin{cases}
    1 & \text{if $i<j$,} \\
    0 & \text{if $i \geq j$.}
  \end{cases}
\]
Then we get 
\begin{align*}
&\models \phi(a_i, c_{\sigma_j}) \quad\text{if} \ i < j, \\
&\models \psi(a_i, c_{\sigma_j}) \quad\text{if} \ i \geq j.
\end{align*}
Therefore, $\phi(x,y)$ has \OP, witnessed by $(a_i)_{i < \omega}$, $(c_{\sigma_j})_{j < \omega}$ and $\psi(x,y)$.
\end{proof}

\section{Interactions with independence relations}
\label{sec:interactions-with-independence-relations}
In this section we study the interaction between independence relations and some of the combinatorial properties studied above. We first recall the notion of dividing and the corresponding definition of simplicity, and the way different notions of independence interact with a theory being \NSOP[1], simple or stable. We do not define Kim-dividing or use it directly; we rely on the axiomatic characterization of the notion of independence given in \cite[Theorem 9.1]{dobrowolski_kim-independence_2022}.
\begin{definition}
\thlabel{def:dividing}
Let $p(x, b) = \tp(a/Cb)$ be a type. We say that $p(x, b)$ \emph{divides} over $C$ if there is a $C$-indiscernible sequence $(b_i)_{i < \omega}$, with $b_0 \equiv_C b$, such that $\bigcup_{i < \omega} p(x, b_i)$ is inconsistent. We write $a \ind^d_C b$ if $\tp(a/Cb)$ does not divide over $C$.
\end{definition}
\begin{lemma}
\thlabel{psi-dividing}
A type $p(x, b) = \tp(a/Cb)$ divides over $C$ if and only if it contains a formula $\phi(x, b) \in p(x, b)$ and there are a negation $\psi(y_1, \ldots, y_k)$ of $\exists x(\phi(x, y_1) \wedge \ldots \wedge \phi(x, y_k))$ and some infinite sequence $(b_i)_{i < \omega}$ such that $b_i \equiv_C b$ for all $i < \omega$ and for all $i_1 < \ldots < i_k < \omega$ we have $\models \psi(b_{i_1}, \ldots, b_{i_k})$.
\end{lemma}
The above lemma is the positive variant of $k$-dividing (see e.g.\ \cite[Definition 7.1.2]{tent_course_2012}). The role of $k$ is replaced by $\psi$, and accordingly we say in the situation of \thref{psi-dividing} that $\phi(x, b)$ \emph{$\psi$-divides over $C$}. The proof is standard, but instructive on how $\psi$ is used.
\begin{proof}
If $p(x, b)$ divides then let $(b_i)_{i < \omega}$ be an indiscernible sequence witnessing this. By compactness there is $\phi(x, b) \in p(x, b)$ and some $k < \omega$ such that $\{\phi(x, b_1), \ldots, \phi(x, b_k)\}$ is inconsistent. So there is a negation $\psi(y_1, \ldots, y_k)$ of $\exists x(\phi(x, y_1) \wedge \ldots \wedge \phi(x, y_k))$ with $\models \psi(b_1, \ldots, b_k)$. It then follows by $C$-indiscernibility that for all $i_1 < \ldots < i_k < \omega$ we have $\models \psi(b_{i_1}, \ldots, b_{i_k})$.

Conversely, suppose that $\phi(x, b) \in p(x, b)$ $\psi$-divides over $C$. Let $(b_i)_{i < \omega}$ be the infinite sequence witnessing this. By compactness we may elongate $(b_i)_{i < \omega}$ to $(b_i)_{i < \lambda}$ for suitably large $\lambda$. Then by \thref{erdos-rado-indiscernible-sequences} we find an indiscernible sequence $(b'_i)_{i < \omega}$ based on $(b_i)_{i < \lambda}$. In particular, there are $i_1 < \ldots < i_k < \lambda$ such that $b'_1 \ldots b'_k \equiv_C b_{i_1} \ldots b_{i_k}$. As we have $\models \psi(b_{i_1}, \ldots, b_{i_k})$, we thus have $\models \psi(b'_1, \ldots, b'_k)$. So $\{\phi(x, b'_1), \ldots, \phi(x, b'_k)\}$ is inconsistent and therefore that $\bigcup_{i < \omega} p(x, b'_i)$ is inconsistent. We conclude that $p(x, b)$ divides over $C$.
\end{proof}
\begin{definition}
\thlabel{def:simple-theory}
We say that a theory $T$ is \emph{simple} if \emph{dividing has local character}. That is, there is some cardinal $\lambda$ such that for any finite $a$ and any parameter set $B$ there is $B_0 \subseteq B$ with $|B_0| \leq \lambda$ such that $\tp(a/B)$ does not divide over $B_0$.
\end{definition}
To make independence work nicely in simple and \NSOP[1] positive theories we need the mild assumption of thickness from \cite{ben-yaacov_thickness_2003}. Note that in particular every theory in full first-order logic, viewed as a positive theory, is thick.
\begin{definition}
\thlabel{def:thickness}
A theory $T$ is called \emph{thick} if being an indiscernible sequence is type-definable. So there is a partial type $\Theta((x_i)_{i < \omega})$ such that $\models \Theta((a_i)_{i < \omega})$ iff $(a_i)_{i < \omega}$ is an indiscernible sequence.
\end{definition}
\begin{definition}
\thlabel{def:lascar-distance-and type}
We write $\d_B(a, a') \leq n$ if there are $a = a_0, a_1, \ldots, a_n = a'$ such that $a_i$ and $a_{i+1}$ are on a $B$-indiscernible sequence for all $0 \leq i < n$. We say that $a$ and $a'$ have the same \emph{Lascar strong type (over $B$)}, and write $a \equivls_B a'$, if $\d_B(a, a') \leq n$ for some $n < \omega$.
\end{definition}
\begin{fact}[{\cite[Lemma 2.20]{dobrowolski_kim-independence_2022}}]
\thlabel{fact:thick-type-over-lambda-t-saturated-model-is-lascar-strong}
Let $T$ be a thick theory and $M$ a $\lambda_T$-saturated e.c.\ model, then $a \equiv_M a'$ implies $a \equivls_M a'$.
\end{fact}
\begin{fact}[{\cite[Theorem 9.1]{dobrowolski_kim-independence_2022}}]
\thlabel{fact:kim-pillay-nsop1}
Let $T$ be a thick theory. Then $T$ is \NSOP[1] if and only if there exists an automorphism invariant ternary relation $\ind$ on subsets, only allowing e.c.\ models in the base, satisfying the following properties:
\begin{description}
\item[\textsc{Finite Character}] if $a \ind_M b_0$ for all finite $b_0 \subseteq b$ then $a \ind_M b$.
\item[\textsc{Existence}] $a \ind_M M$ for any e.c.\ model $M$.
\item[\textsc{Monotonicity}] $a a' \ind_M b b'$ implies $a \ind_M b$.
\item[\textsc{Symmetry}] $a \ind_M b$ implies $b \ind_M a$.
\item[\textsc{Chain Local Character}] let $a$ be a finite tuple and $\kappa > |T|$ be regular then for every continuous chain $(M_i)_{i < \kappa}$, with $|M_i| < \kappa$ for all $i$, there is $i_0 < \kappa$ such that $a \ind_{M_{i_0}} M$, where $M = \bigcup_{i < \kappa} M_i$.
\item[\textsc{Independence Theorem}] if $a \ind_M b$, $a' \ind_M c$ and $b \ind_M c$ with $a \equivls_M a'$ then there is $a''$ such that $a''b \equivls_M ab$, $a''c \equivls_M a'c$ and $a'' \ind_M bc$.
\item[\textsc{Extension}] if $a \ind_M b$ then for any $c$ there is $a' \equiv_{Mb} a$ such that $a' \ind_M bc$.
\item[\textsc{Transitivity}] if $a \ind_M N$ and $a \ind_N b$ with $M \subseteq N$ then $a \ind_M Nb$.
\end{description}
Furthermore, in this case $\ind = \ind^K$ is given by non-Kim-dividing.
\end{fact}
\begin{fact}[{\cite[Theorem 1.51]{ben-yaacov_simplicity_2003} and \cite[Theorem 1.15]{ben-yaacov_thickness_2003}}]
\thlabel{fact:kim-pillay-simple}
Let $T$ be a thick theory. Then $T$ is simple if and only if there exists an automorphism invariant ternary relation $\ind$ on subsets, only allowing e.c.\ models in the base, satisfying all the properties from \thref{fact:kim-pillay-nsop1} as well as:
\begin{description}
\item[\textsc{Base-Monotonicity}] if $a \ind_M B$ and $M \subseteq N \subseteq B$, with $N$ an e.c.\ model, then $a \ind_N B$.
\end{description}
Furthermore, in this case $\ind = \ind^d$ is given by non-dividing.

In this fact we may restrict the base of $\ind$ further to $\kappa$-saturated e.c.\ models for some fixed $\kappa$.
\end{fact}
\begin{fact}[{\cite[Theorem 2.8]{ben-yaacov_simplicity_2003}}]
\thlabel{fact:stable-iff-stationary-types}
Let $T$ be a thick theory. Then $T$ is stable if and only if it is simple and dividing independence satisfies \textsc{Stationarity} over $\lambda_T$-saturated e.c.\ models: whenever $M$ is a $\lambda_T$-saturated e.c.\ model, $a \ind^d_M  b$, $a' \ind^d_M b$ and $a \equiv_M a'$ then $a \equiv_{Mb} a'$.
\end{fact}
In \thref{fact:stable-iff-stationary-types} we use $\lambda_T$-saturated e.c.\ models, because we want types over these e.c.\ models to be Lascar strong types (see \thref{fact:thick-type-over-lambda-t-saturated-model-is-lascar-strong}). The proof of \cite[Theorem 2.8]{ben-yaacov_simplicity_2003} works with $|T|^+$-saturated e.c.\ models, but goes through for $\lambda_T$-saturated e.c.\ models as well (noting that $\lambda_T > |T|^+$).
\begin{proposition}
\thlabel{prop:stable-iff-stationarity}
A thick \NSOP[1] theory $T$ is stable iff Kim-independence satisfies \textsc{Stationarity} over $\lambda_T$-saturated e.c.\ models.
\end{proposition}
\begin{proof}
If $T$ is stable then Kim-dividing is the same as dividing by the canonicity parts of \thref{fact:kim-pillay-nsop1} and \thref{fact:kim-pillay-simple}, so by \thref{fact:stable-iff-stationary-types} we have \textsc{Stationarity} over $\lambda_T$-saturated e.c.\ models for Kim-dividing.

By the other direction of \thref{fact:stable-iff-stationary-types} it suffices to prove that \textsc{Stationarity} for $\ind^K$ implies that $T$ is simple. By \thref{fact:kim-pillay-simple} it is then enough to prove that $\ind^K$ satisfies \textsc{Base-Monotonicity}, where we may in fact restrict ourselves in the base to $\lambda_T$-saturated e.c.\ models. So let $M$ be a $\lambda_T$-saturated e.c.\ model such that $a \ind_M^K B$, and let $N$ be a ($\lambda_T$-saturated) e.c.\ model such that $M \subseteq N \subseteq B$. By \textsc{Existence} we have $a \ind_N^K N$, so by \textsc{Extension} we find $a' \equiv_N a$ with $a' \ind_N^K B$. By \textsc{Monotonicity} applied to $a \ind_M^K B$ we find $a \ind_M^K N$ and so $a' \ind_M^K N$. We then apply \textsc{Transitivity} to find $a' \ind_M^K B$. As $a \equiv_M a'$ we can apply \textsc{Stationarity} to find $a \equiv_B a'$ and so we conclude $a \ind_N^K B$, as required.
\end{proof}
\begin{definition}
\thlabel{def:kim-morley-sequence}
Let $M$ be some e.c.\ model. An \emph{$\ind^K_M$-Morley sequence} is an $M$-indiscernible sequence $(a_i)_{i < \omega}$ such that $a_i \ind^K_M (a_j)_{j < i}$ for all $i < \omega$.
\end{definition}
\begin{lemma}
\thlabel{nsop1-construct-kim-morley-sequence}
Let $T$ be a thick \NSOP[1] theory, and let $a$ be any tuple and $M$ be any e.c.\ model. Then there is an $\ind^K_M$-Morley sequence $(a_i)_{i < \omega}$ with $a_0 = a$.
\end{lemma}
\begin{proof}
Standard, but we give the proof for completeness. By \textsc{Existence} we have $a \ind^K_M M$. So by repeatedly applying \textsc{Extension} we find, for some big enough $\lambda$, a sequence $(a'_i)_{i < \lambda}$ with $a'_i \equiv_M a$ and $a'_i \ind^K_M (a'_j)_{j < i}$ for all $i < \lambda$. Then using \thref{erdos-rado-indiscernible-sequences} we base an $M$-indiscernible sequence $(a_i)_{i < \omega}$ on $(a'_i)_{i < \lambda}$, and by an automorphism we may assume $a_0 = a$. By \textsc{Finite Character} it is then enough to verify that for any $i_1 < \ldots < i_n < \omega$ we have $a_{i_n} \ind^K_M a_{i_1} \ldots a_{i_{n-1}}$, which follows because there are $j_1 < \ldots < j_n < \lambda$ such that $a_{i_1} \ldots a_{i_n} \equiv_M a'_{j_1} \ldots a'_{j_n}$.
\end{proof}
\begin{repeated-theorem}[\thref{thm:op-iff-ip-or-sop1}]
A thick theory $T$ has \OP iff it has \IP or \SOP[1]. Equivalently: $T$ is stable iff it is \NIP and \NSOP[1].
\end{repeated-theorem}
\begin{proof}
From \thref{implication-diagram} we already know that \IP and \SOP[1] imply \OP, i.e.\ unstability. For the other direction we will prove that any thick unstable \NSOP[1] theory has \IP.

As $T$ is unstable we have by \thref{prop:stable-iff-stationarity} that there is a $\lambda_T$-saturated e.c.\ model $M$ such that \textsc{Stationarity} over $M$ fails. That is, there are $a_0, a_1, b$ such that $a_0 \ind_M^K b$, $a_1 \ind_M^K b$ and $a_0 \equiv_M a_1$ while $a_0 \not \equiv_{Mb} a_1$. Write $p_0(x, y) = \tp(a_0 b / M)$ and $p_1(x, y) = \tp(a_1 b / M)$. Use \thref{nsop1-construct-kim-morley-sequence} to find an $\ind^K_M$-Morley sequence $(b_i)_{i < \omega}$ with $b_0 = b$. We will now construct $(a_\eta)_{\eta \in 2^{< \omega}}$ by induction on the length (domain) of $\eta$, such that for $\eta \in 2^n$:
\begin{enumerate}
\item $\models p_{\eta(i)}(a_\eta, b_i)$ for all $i < n$,
\item $a_\eta \ind_M^K b_{<n}$,
\item $a_\eta \equivls_M a_0 \equivls_M a_1$.
\end{enumerate}
For $a_{\langle 0 \rangle}$ and $a_{\langle 1 \rangle}$ we can just take $a_0$ and $a_1$, respectively, where (3) is satisfied by \thref{fact:thick-type-over-lambda-t-saturated-model-is-lascar-strong} and the fact that $M$ is $\lambda_T$-saturated. Now assume that $(a_\eta)_{\eta \in 2^{\leq n}}$ has been constructed and let $\eta \in 2^{n+1}$. As $b_n \equivls_M b_0 = b$ we can find $a'$ such that $a' b_n \equivls_M a_{\eta(n)} b$. We also have $a' \ind_M^K b_n$, $a_{\eta|_n} \ind_M^K b_{<n}$ and $b_n \ind_M^K b_{<n}$, so by \textsc{Independence Theorem} we find the required $a_\eta$, where (1)--(3) are easily verified using the induction hypothesis and the application of \textsc{Independence Theorem}.

By (1) and compactness we now find $(a_\sigma)_{\sigma \in 2^\omega}$ such that $a_\sigma \models \bigcup_{i < \omega} p_{\sigma(i)}(x, b_i)$ for every $\sigma \in 2^\omega$. As $p_0(x, y)$ and $p_1(x, y)$ are distinct, there is $\phi(x, y) \in p_1(x, y)$ such that $\phi(x, y) \not \in p_0(x, y)$. So there is a negation $\psi(x, y)$ of $\phi(x, y)$ with $\psi(x, y) \in p_0(x, y)$. Now for any $\sigma \in 2^\omega$ and $i < \omega$ we have:
\[
\begin{array}{rcccl}
\sigma(i) = 1 & \implies & \models p_1(a_\sigma, b_i) & \implies & \models \phi(a_\sigma, b_i), \\
\sigma(i) = 0 & \implies & \models p_0(a_\sigma, b_i) & \implies & \models \psi(a_\sigma, b_i).
\end{array}
\]
Any parameters from $M$ contained in $\phi(x, y)$ or $\psi(x, y)$ can be assumed to be part of the $b_i$'s, so we see that the formula $\theta(y, x) := \phi(x, y)$ has \IP.
\end{proof}
\begin{repeated-theorem}[\thref{thm:tp-iff-sop1-or-tp2}]
A thick theory $T$ has \TP iff it has \SOP[1] or \TP[2]. Equivalently: $T$ is simple iff it is \NSOP[1] and \NTP[2].
\end{repeated-theorem}
\begin{proof}
From \thref{implication-diagram} we know that if $T$ has \SOP[1] or \TP[2] then it has \TP. We will prove the converse by proving that a thick non-simple \NSOP[1] theory has \TP[2]. This uses that simplicity is equivalent to \NTP, which is exactly \thref{dividing-local-character-iff-ntp}.

Assume then that $T$ is thick, non-simple, and \NSOP[1]. By \thref{fact:kim-pillay-nsop1} and \thref{fact:kim-pillay-simple} we have that an \NSOP[1] theory is simple iff $\ind^K$ satisfies \textsc{Base-Monotonicity}. We have that $\ind^d$ always satisfies \textsc{Base-Monotonicity} by definition: if $\tp(a/B)$ does not divide over $M$ and $M \subseteq N$ are e.c.\ models contained in $B$ then $\tp(a/B)$ does not divide over $N$. Hence we must have $\ind^d \neq \ind^K$. It follows easily from the definition of Kim-dividing (see for example \cite[Remark 4.12]{dobrowolski_kim-independence_2022}) that $\ind^d \implies \ind^K$, so there must be $a, b, M$ with $a \ind^K_M b$ while $a \nind^d_M b$. Write $p(x, b) = \tp(a/Mb)$ and let $J = (b_j)_{j < \omega}$ be an $M$-indiscernible sequence with $b_0 = b$ such that $\bigcup_{j < \omega} p(x, b_j)$ is inconsistent. So there is $\phi(x, y) \in p(x, y)$ together with a negation $\psi(y_1, \ldots, y_k)$ of $\exists x(\phi(x, y_1) \wedge \ldots \wedge \phi(x,y_k))$ such that for all $j_1 < \ldots < j_k < \omega$ we have $\models \psi(b_{j_1}, \ldots, b_{j_k})$. We claim that $\phi$ has $k$-\TP[2], as witnessed by $\psi$.

By \thref{nsop1-construct-kim-morley-sequence} we find an $\ind^K_M$-Morley sequence $(J_i)_{i < \omega}$ with $J_0 = J$. For $i < \omega$ we write $J_i = (c_{i,j})_{j < \omega}$. This yields an array $(c_{i,j})_{i,j < \omega}$ such that the following hold.
\begin{enumerate}[label=(\arabic*)]
\item For all $\sigma \in \omega^\omega$ the set $\{\phi(x, c_{i,\sigma(i)}) : i < \omega\}$ is consistent. First note that for any $i < \omega$ we have $c_{i,\sigma(i)} \equivls_M c_{i,0} \equivls_M c_{0, 0} = b$. So $(c_{i,\sigma(i)})_{i < \omega}$ is an $\ind^K_M$-independent sequence, all having the same Lascar strong type as $b$ over $M$, hence by the usual inductive application of compactness and \textsc{Independence Theorem} we get that $\{\phi(x, c_{i,\sigma(i)}) : i < \omega\}$ is consistent.
\item For all $i < \omega$ and $j_1 < \ldots < j_k < \omega$ we have $\models \psi(c_{i,j_1}, \ldots, c_{i,j_k})$. This follows because $J_i$ is an $M$-automorphic copy of $J_0 = J$.
\end{enumerate}
Any parameters from $M$ contained in $\phi$ or $\psi$ can be assumed to be part of the $c_{i,j}$'s, so we conclude that $\phi$ does indeed have $k$-\TP[2].
\end{proof}
\begin{remark}
\thlabel{rem:two-famous-theorems-comparison}
Compared to \thref{thm:first-order-op-iff-ip-or-sop} we replaced \SOP by \SOP[1] in \thref{thm:op-iff-ip-or-sop1}, which gives a weaker result. However, it is not even clear how the \SOP property should be formulated in positive logic, see also \thref{rem:sop}. In the \thref{thm:tp-iff-sop1-or-tp2} we replaced \TP[1] by \SOP[1], compared to \thref{thm:first-order-tp-iff-tp1-or-tp2}, again giving an a priori slightly weaker result. However, for full first-order logic Mutchnik's recent preprint \cite{mutchnik_nsop_2_2022} proves that \SOP[1] is equivalent to \TP[1], and it is very reasonable to expect the same thing in positive logic, see also \thref{q:sop1-iff-sop2}.

It is worth noting that the proofs of both theorems here are completely different from the classical proofs for \thref{thm:first-order-op-iff-ip-or-sop} and \thref{thm:first-order-tp-iff-tp1-or-tp2}. In particular, modulo Mutchnik's result, \thref{thm:tp-iff-sop1-or-tp2} gives a completely new proof of \thref{thm:first-order-tp-iff-tp1-or-tp2}.
\end{remark}
For the remainder of this section we shift our focus to the equivalent definitions of a simple theory.
\begin{theorem}
\thlabel{dividing-local-character-iff-ntp}
A theory $T$ is simple if and only if it does not have \TP.
\end{theorem}
\begin{proof}
\underline{\TP $\implies$ failure of local character.}
Take any cardinal $\lambda$ and suppose that $\phi(x, y)$ has \TP. By compactness we can assume this is witnessed by $(c_\eta)_{\eta\in \kappa^{< \lambda^+}}$ for $\kappa = (2^{|T| + \lambda^+})^+$ and $\psi(y_1, \ldots, y_k)$.

We construct some $\sigma \in \kappa^{\lambda^+}$ by induction on its length (i.e.\ its domain). Suppose we have already defined $\sigma|_\gamma$ for some $\gamma < \lambda^+$. Write $C = \{c_{\sigma|_i} : i \leq \gamma\}$ and $\eta = \sigma|_{\gamma}$. Consider the set of types $\{\tp(c_{\eta^\frown i} / C) : i < \kappa\}$. There are at most $2^{|T| + \lambda^+}$ different types over $C$, so by our choice of $\kappa$ and the pigeonhole principle there exists infinite $I_\gamma \subseteq \kappa$ such that for any $i, j \in I_\gamma$ we have $\tp(c_{\eta^\frown i} / C) = \tp(c_{\eta^\frown j} / C)$. Let $i_0$ be the least element of $I_\gamma$ and define $\sigma(\gamma) = i_0$.

Now that the construction of $\sigma$ is finished we write $b_i = c_{\sigma|_i}$ for $i < \lambda^+$. With this notation, and using the fact that $(c_\eta)_{\eta\in \kappa^{< \lambda^+}}$ witnesses \TP for $\phi(x, y)$, we find some $a$ realising $\{\phi(x, b_i) : i < \lambda^+ \}$. We claim that $\tp(a / (b_i)_{i < \lambda^+})$ divides over every subset $B_0 \subseteq (b_i)_{i<\lambda^+}$ with $|B_0| \leq \lambda$. Suppose for a contradiction that $\tp(a / (b_i)_{i < \lambda^+})$ does not divide over some $B_0 \subseteq (b_i)_{i<\lambda^+}$, where $|B_0| \leq \lambda$. Let $\gamma < \lambda^+$ such that $B_0 \subseteq (b_i)_{i<\gamma}$. Then $\tp(a / (b_i)_{i < \lambda^+})$ does not divide over $(b_i)_{i<\gamma}$ by \textsc{Base-Monotonicity} (which holds for dividing in any theory, as already mentioned in the proof of \thref{thm:tp-iff-sop1-or-tp2}). We have $\phi(x, b_\gamma) \in \tp(a / (b_i)_{i<\lambda^+})$, hence by \thref{psi-dividing} it suffices to prove that $\phi(x, b_\gamma)$ $\psi$-divides over $(b_i)_{i<\gamma}$. Enumerate the set $I_\gamma$ from the construction of $\sigma$ as $i_0 < i_1 < \ldots$ and let $d_j = c_{{\sigma|_\gamma}^\frown i_j}$ for $j < \omega$. Note that $d_0 = b_\gamma$. Then by the construction of $I_\gamma$, $(d_j)_{j < \omega}$ is a sequence of realizations of $\tp(b_\gamma / (b_i)_{i<\gamma})$. Moreover, by \TP we also have $\models \psi (d_{j_1}, \ldots, d_{j_k})$ for all \mbox{$j_1 < \ldots < j_k < \omega$}. Therefore, $\tp(a/(b_i)_{i<\lambda^+})$ divides over $B_0$ and $T$ does not have local character.

\underline{Failure of local character $\implies$ \TP.}
Let $\lambda = |T|^+$. As local character for dividing fails there is some finite $a$ and a parameter set $B$ such that $\tp(a/B)$ divides over $B_0$ for every $B_0 \subseteq B$ with $|B_0| \leq \lambda$.

We construct a tree $(c_\eta)_{\eta \in \omega^{< \lambda}}$ by induction on its height. Let $\zeta_\alpha \in \omega^\alpha$ denote the constant zero function. As induction hypothesis for step $\delta$ we use three statements:
\begin{itemize}
    \item $c_{\zeta_\alpha}$ is a finite tuple of elements from $B$ for all $\alpha \leq \delta$;
    \item $(c_{\eta|_\alpha})_{\alpha \leq \delta} \equiv (c_{\zeta_\alpha})_{\alpha \leq \delta}$ for all $\eta \in \omega^\delta$;
    \item if $\delta = \gamma + 1$ is a successor then there are $\phi_\delta(x, y)$ and a negation $\psi_\delta(y_1, \ldots, y_{k_\delta})$ of $\exists x (\phi_\delta(x, y_1) \wedge \ldots \wedge \phi_\delta(x, y_{k_\delta}))$ such that for any $\eta \in \omega^\gamma$ and any $i_1 < \ldots < i_{k_\delta} < \omega$ we have $\models \psi_\delta(c_{\eta^\frown i_1}, \ldots, c_{\eta^\frown i_{k_\delta}})$.
\end{itemize}
For $\delta < \lambda$ limit or zero we let all $c_\eta$, where $\eta \in \omega^\delta$, be the empty tuple. Now suppose that we constructed $(c_\eta)_{\eta \in \omega^{\leq \delta}}$ and we need to construct level $\delta+1$. As $c_{\zeta_\alpha}$ is a finite tuple for all $\alpha \leq \delta$ and $\delta < \lambda$, we have that $|\{c_{\zeta_\alpha} : \alpha \leq \delta \}| < \lambda$ and so $\tp(a/B)$ divides over $\{c_{\zeta_\alpha} : \alpha \leq \delta \}$. By \thref{psi-dividing} there exists formulas $\phi_{\delta+1}(x, d) \in \tp(a/B)$ and a negation $\psi_{\delta+1}(y_1, \ldots, y_{k_{\delta+1}})$ of $\exists x (\phi_{\delta+1}(x, y_1) \wedge \ldots \wedge \phi_{\delta+1}(x, y_{k_{\delta+1}}))$ together with a sequence $(d_i)_{i< \omega}$ such that $d_i \equiv_{(c_{\zeta_\alpha})_{\alpha \leq \delta}} d$ for all $i < \omega$ and for any $i_1 < \ldots < i_{k_{\delta+1}} < \omega$ we have $\models \psi_{\delta+1}(d_{i_1}, \ldots, d_{i_{k_{\delta+1}}})$. By an automorphism we may assume $d_0 = d$. Note that $d_0$ is a finite tuple of elements from $B$. We start by defining $c_{\zeta_\delta^\frown i}$ to be $d_i$ for all $i < \omega$. Since $d_i \equiv_{(c_{\zeta_\alpha})_{\alpha \leq \delta}} d_j$ for all $i, j < \omega$, we have $(c_{{\zeta_\delta^\frown i}|_\alpha})_{\alpha \leq \delta+1} \equiv (c_{\zeta_\alpha})_{\alpha \leq \delta+1}$. Now take any $\eta \in \omega^\delta$. We know by induction hypothesis that $(c_{\eta|_\alpha})_{\alpha \leq \delta} \equiv (c_{\zeta_\alpha})_{\alpha \leq \delta}$. Let $(c_{\eta^\frown i})_{i < \omega}$ be such that $(c_{\eta^\frown i})_{i < \omega} (c_{\eta|_\alpha})_{\alpha \leq \delta} \equiv (d_i)_{i < \omega} (c_{\zeta_\alpha})_{\alpha \leq \delta}$. Then the induction hypothesis holds by construction. Moreover, since for any $i_1 < \ldots < i_{k_{\delta+1}} < \omega$ we have $\models \psi_{\delta+1}(d_{i_1}, \ldots, d_{i_{k_{\delta+1}}})$, we now also get $\models \psi_{\delta+1}(c_{\eta^\frown i_1}, \ldots, c_{\eta^\frown i_{k_{\delta+1}}})$. This completes the inductive construction of the tree $(c_\eta)_{\eta \in \omega^{< \lambda}}$.

There are $|T|$ possible pairs of formulas $\phi(x,y)$ and $\psi(y_1, \ldots, y_k)$ but we have $\lambda = |T|^+$ successor levels, and each is assigned a pair $\phi_\delta(x,y)$ and $\psi_\delta(y_1, \ldots, y_{k_\delta})$. Hence, by pigeonhole principle we can choose an infinite set of successor levels $l_0 < l_1 < l_2 < \ldots$ having the same $\phi_\delta$ and $\psi_\delta$. We denote these just as $\phi(x,y)$ and $\psi(y_1, \ldots, y_k)$. We consider a subtree $(f_\mu)_{\mu \in \omega^{< \omega}}$ that consists only of the chosen levels (with the root being the leftmost point on level $l_0$). That is, for $\mu \in \omega^{< \omega}$ of length $n$ we define $\eta_\mu \in \omega^{l_n}$ of length $l_n$ as 
\[
\eta_\mu(l) =
    \begin{cases}
        \mu(i) & \text{if } l = l_{i+1} - 1\\
        0 & \text{otherwise.}
    \end{cases}
\]
Note that $l_{i+1} - 1$ makes sense, because we only chose successor levels. Let $f_\mu = c_{\eta_\mu}$.

We claim $(f_\mu)_{\mu \in \omega^{< \omega}}$ and $\psi$ witness \TP for $\phi$. Let $\sigma \in \omega^\omega$. By construction $\{\phi (x, f_{\zeta_n}) : n < \omega\} \subseteq \tp(a/B)$, and so this set is consistent. Then since $(f_{\sigma|_n})_{n < \omega} \equiv (f_{\zeta_n})_{n < \omega}$ we get that $\{\phi (x, f_{\sigma|_n}) : n < \omega\}$ is consistent. Finally take any $\mu \in \omega^{< \omega}$ and $i_1 < \ldots < i_k < \omega$. The elements $f_{\mu^\frown i_1}, \ldots, f_{\mu^\frown i_k}$ are equal to $c_{\eta^\frown i_1}, \ldots, c_{\eta^\frown i_k}$ for some $\eta \in \omega^{< \lambda}$. Hence, by construction of the subtree, we get $\models \psi(f_{\mu^\frown i_1}, \ldots, f_{\mu^\frown i_k})$. We conclude that $\phi$, and thus $T$, has \TP.
\end{proof}
\begin{remark}
\thlabel{rem:local-character-dividing-vs-forking}
In \cite{ben-yaacov_simplicity_2003} simplicity of a theory is defined as dividing having local character, as we did here. In \cite{pillay_forking_2000} simplicity is defined as forking having local character. Note that trivially local character of forking implies local character of dividing. In thick theories the converse is true: by \cite[Theorem 1.15]{ben-yaacov_thickness_2003} non-dividing satisfies \textsc{Extension} and so dividing coincides with forking.

Without the thickness assumption the converse can fail. By \cite[Example 4.3]{ben-yaacov_simplicity_2003} there is a stable positive theory $T$ with a type over the empty set that forks over the empty set.

The above example motivates our choice of terminology for simplicity, because if we defined simplicity in terms of local character for forking then stability would not imply simplicity. Furthermore, the fact that local character of dividing is equivalent to \NTP (\thref{dividing-local-character-iff-ntp}) does not need thickness. So the $T$ from above is an example of an \NTP theory where local character for forking fails, further motivating our choice of terminology.
\end{remark}

\section{Further discussion and open questions}
\label{sec:discussion-and-open-questions}
In light of the recent \cite{mutchnik_nsop_2_2022}, where it is shown that \SOP[1] is equivalent to \SOP[2] for theories in full first-order logic, the following is a natural question.
\begin{question}
\thlabel{q:sop1-iff-sop2}
Is \SOP[1] equivalent to \SOP[2] in positive logic? As Mutchnik's proof \cite{mutchnik_nsop_2_2022} makes heavy use of various notions of independence, and these tend to work better in thick theories \cite{dobrowolski_kim-independence_2022, ben-yaacov_thickness_2003}, it would be natural to assume thickness in order to answer this question.
\end{question}
\begin{question}
\thlabel{q:sop3-iff-sop2}
Is \SOP[3] equivalent to \SOP[2], and so, if \thref{q:sop1-iff-sop2} has a positive answer, also equivalent to \SOP[1]?
\end{question}
\begin{remark}
\thlabel{q:2-tp2-iff-k-tp2}
In full first-order logic we have that if $\phi(x, y)$ has $k$-\TP[2] for some $k \geq 2$ then some conjunction $\bigwedge_{i = 1}^n \phi(x, y_i)$ has $2$-\TP[2] \cite[Propostion 5.7]{kim_tree_2014}. The proof for this makes use of array-indiscernibles and array-modelling. The development of these tools is out of the scope of this paper, but it is done in a recent preprint \cite{kamsma_positive_2023} by the third author. In particular, \cite[Theorem 1.4]{kamsma_positive_2023} proves the above statement for thick theories. This implies in particular that a thick theory has $2$-\TP[2] iff it has $k$-\TP[2] for some $k \geq 2$.
\end{remark}
\begin{remark}
\thlabel{rem:sop}
In this work we left out the \SOP[n] hierarchy for $n \geq 4$, as well as the \emph{strict order property} \SOP. It is not clear what would be the right definition of these properties in positive logic. The combinatorial properties we have considered all have a similar form: there is some combinatorial configuration of parameters and we require a formula to be consistent along certain parts of those parameters, while being inconsistent along other parts. The only change for positive logic is then that we require this inconsistency to be uniformly witnessed by some negation. However, \SOP[\geq 4] and \SOP are defined in a different way and, unlike \SOP[3], there is no known equivalent formulation of the above form.
\end{remark}
\begin{remark}
\thlabel{rem:nip}
The first work to consider the independence property \IP in positive logic is \cite{dobrowolski_amalgamation_2023}. There some basics for positive \NIP theories are developed, such as closure of \NIP formulas under conjunctions and disjunctions and the fact that one can swap the roles of the variables. On the level of theories they also prove that to verify that a theory is \NIP one only needs to check the formulas $\phi(x, y)$ where $y$ is a single variable (as opposed to a tuple of variables).
\end{remark}
\begin{remark}
\thlabel{rem:preservation-of-combinatorial-properties-under-hyperimaginaries}
In positive logic we can add hyperimaginaries (e.g.\ the $(-)^\text{heq}$-construction) in the same way we can add imaginaries (e.g.\ the $(-)^\text{eq}$-construction) in full first-order logic, see \cite[Subsection 10C]{dobrowolski_kim-independence_2022} for details. In \cite[Theorem 10.18]{dobrowolski_kim-independence_2022} it is proved that whether a theory is \SOP[1] or \NSOP[1] is preserved under such hyperimaginary extensions. As is remarked there as well, the proof strategy should go through for any of the combinatorial properties discussed in this paper. For \NIP the details of this are verified in \cite[Proposition 6.22]{dobrowolski_amalgamation_2023}.
\end{remark}

% References
\bibliographystyle{alpha}
\bibliography{bibfile}

\end{document}